%% file: article.tex
 \documentclass{degruyter-journal-a}           

	\usepackage{graphics}
	\usepackage{amssymb, amsmath, amsfonts}
	\usepackage{tikz}
	\usepackage{float}
	\usepackage{mathabx}
	\usepackage{stmaryrd}
	\usepackage{rotating}
	\usepackage{graphicx}
	\usepackage{caption}
	\pgfdeclarelayer{edgelayer}
\pgfdeclarelayer{nodelayer}
\pgfsetlayers{edgelayer,nodelayer,main}

\tikzstyle{none}=[inner sep=0pt]
\tikzstyle{nfilled}=[circle,draw=black,fill=white,inner sep=0pt,minimum size=2mm]
\tikzstyle{filled}=[circle,draw=black,fill=black,inner sep=0pt,minimum size=2mm]

\title{Fictitious domain method with boundary value correction using penalty-free Nitsche method}
\headlinetitle{Boundary value correction, penalty-free Nitsche method}

\lastnameone{Boiveau}
\firstnameone{Thomas}
\nameshortone{T.~Boiveau}
\addressone{Department of Mathematics, University College London, Gower Street, London, UK-WC1E  6BT}
\countryone{United Kingdom}
\emailone{thomas.boiveau.12@ucl.ac.uk}

\lastnametwo{Burman}
\firstnametwo{Erik}
\nameshorttwo{E.~Burman}
\addresstwo{Department of Mathematics, University College London, Gower Street, London, UK-WC1E  6BT}
\countrytwo{United Kingdom}
\emailtwo{}

\lastnamethree{Claus}
\firstnamethree{Susanne}
\nameshortthree{S,~Claus}
\addressthree{Cardiff School of Engineering, Cardiff University, Queen's Buildings, The Parade, Cardiff CF24 3AA, Wales}
\countrythree{United Kingdom}
\emailthree{}

\lastnamefour{Larson}
\firstnamefour{Mats~G.}
\nameshortfour{M.\ G.~Larson}
\addressfour{Department of Mathematics and Mathematical Statistics, Ume\aa~University, SE-90187 Ume\aa}
\countryfour{Sweden}
\emailfour{}

\abstract{In this paper, we consider a fictitious domain approach
  based on a Nitsche type method without penalty. To allow for high
  order approximation using piecewise affine approximation of the
  geometry we use a boundary value correction technique based on
  Taylor expansion from the approximate to the physical boundary. To
  ensure stability of the method a ghost penalty stabilization is
  considered in the boundary zone. We prove optimal error estimates in
  the $H^1$-norm and estimates suboptimal by $\mathcal{O}(h^{\frac12})$ in the
  $L^2$-norm. The suboptimality is due to the lack of adjoint
  consistency of our formulation. Numerical results are provided to corroborate the theoretical study.}

\keywords{Nitsche's method, fictitious domain method, boundary value correction}

\classification{65N12, 65N30, 65N85}

\researchsupported{
This work received funding from EPSRC (award number EP/J002313/1) which is gratefully acknowledged. The Author, S. Claus, gratefully acknowledges the financial support provided by the Welsh Government and Higher Education Funding Council for Wales through the Sêr Cymru National Research Network in Advanced Engineering and Materials. 
The author M. G. Larson gratefully acknowledges the financial support  from the Swedish Foundation for Strategic Research Grant No.\ AM13-0029, the Swedish Research Council Grant No.\ 2011-4992, and Swedish strategic research programme eSSENCE.
}

\begin{document}

\section{Introduction}
Mesh generation is an important challenge in computational mechanics,
in fact for complex geometries this can be highly nontrivial. In some
cases for time dependent problems, such as a solid body embedded in a
flow, the geometry of the problem changes each time step imposing
conintuous remeshing, at least locally. The main idea of the
fictitious domain method \cite{Girault_1995_a, Girault_1999_a,
  Hansbo_2002_a, Angot_2005_a, Haslinger_2009_a, Burman_2010_b,
  Burman_2012_a, Massing_2014_a} is to relax the constraint that
imposes the mesh to fit with the computational domain. In fact the
principle is to embed the computational domain in a mesh that is easy
to generate, without matching the elements with the boundary. In the
early developments of fictitious domain \cite{Girault_1995_a}, the
method was faced with the choice of either integrating the equations
over the whole computational mesh including the nonphysical part or
only integrate inside the physical domain. In the first case, the
method is robust but inaccurate, the second approach is accurate but
can generate bad conditioning of the system matrix depending on how
the boundary crosses the mesh. As a fix to solve the conditioning
problem a boundary penalty term was introduced in \cite{Burman_2010_a}
the effect of this term is that it extends the stability in the
physical domain to the whole mesh domain, provided the distance from
the mesh boundary to the physical boundary is $\mathcal{O}(h)$.

Nitsche's method was first introduced for the weak imposition of the boundary conditions in \cite{Nitsche_1971_a} and
designed to be consistent and preserve the symmetry of the original
problem. The stability of the method relies on a penalty term that
needs to be sufficiently large. In the context of fictitious domain methods
Nitsche's method can suffer from instability for certain mesh boundary
configurations. A solution to this problem using the ghost penalty
approach was suggested in \cite{Burman_2012_a, Massing_2014_a}.

 In \cite{Freund_1995_a} a
non-symmetric version was proposed where the penalty parameter only
needs to be strictly greater than zero for the stability to be
ensured.
The possibility of considering the penalty parameter equal to
zero for the non-symmetric case was suggested in \cite{Hughes_2000_a},
however, coercivity cannot be proved for this non-symmetric
penalty-free method and stability was not established. In
\cite{Burman_2012_b} the stability of the nonsymmetric Nitsche's
method without penalty for elliptic problems was proved, using an
inf-sup argument, drawing on earlier work on discontinuous Galerkin
methods \cite{LN04}. Recently the work on penalty free methods has
been extended to compressible and incompressible elasticity in
\cite{Boiveau_2015_a}. The penalty-free method can be seen as a
Lagrange multiplier method where the Lagrange multipliers has been
replaced by the the boundary fluxes of the discrete elliptic
operator. In multiphysics problems and particularly for
fluid-structure interaction the loose coupling of this method appears to have some advantages that has been observed numerically in \cite{Burman_2014_c}.

We consider a cut finite element method (CutFEM) \cite{Burman_2014_g}
in the fictitious domain fashion, the implementation of this method
often requires an approximation of the physical domain due to the
boundary that can arbitrarily cut through the elements of the mesh. In
this paper we propose a method to control the error introduced by this
approximation of the physical domain. We follow the method that has
been developed in \cite{Burman_2015_a} where a piecewise affine
approximate geometry was used for the integration of the equation with a
correction based on Taylor expansion from the approximate to the
physical boundary to improve order when higher polynomial orders are
used. In this work we use the penalty-free Nitsche's method to impose the boundary conditions.
This eliminates one penalty parameter at the price of loss of symmetry
of the algebraic system and $\mathcal{O}(h^{\frac12})$ suboptimality in
the $L^2$-norm. We believe that the method nevertheless may be of
interest, in
particular for problems that are not symmetric, such as the
advection--diffusion equation. We present and analyse the method in
the two-dimensional case, but the results hold also in the three
dimensional case.

We end this section by introducing our model problem. Let $\Omega$ be a bounded domain in $\mathbb{R}^2$ with smooth boundary $\Gamma$ and exterior unit normal $n$. The Poisson problem is given by:
\begin{equation}
\begin{aligned}
\label{poisson}
-\Delta u&=f\quad\text{in}~\Omega,\nonumber\\
u&=g\quad\text{on}~\Gamma,
\end{aligned}
\end{equation}
where $f\in L^2(\Omega)$ the given body force and $g\in H^{\frac32}(\Omega)$ the boundary condition. The following regularity estimate holds
\begin{equation}
\label{regularity}
\|u\|_{H^{s+2}(\Omega)}\lesssim \|f\|_{H^s(\Omega)}~~~,~~~~~~s\geq -1.
\end{equation}
In this paper $C$ will be used as a generic positive constant that may
change at each occurrence, we use the notation $a\lesssim b$ for
$a\leq C b$. We will also use $a \sim b$ to denote $a \lesssim b$ and
$b \lesssim a$.
\section{Preliminaries}
\label{preliminaries}
Let $\left\{\mathcal{T}_h\right\}_h$ be a family of quasi-uniform and shape regular triangulations. In a generic sense a node of the triangulation is designated by $x_i$, $K$ denotes a triangle of $\mathcal{T}_h$ and $F$ denotes a face of a triangle $K$.
$h_K=\mbox{diam}(K)$ is the diameter of $K$ and $h=\mbox{max}_{K\in\mathcal{T}_h}h_K$ the mesh parameter for a given triangulation $\mathcal{T}_h$. $\mathbb{P}_p(K)$ defines the space of polynomials of degree less than or equal to $p$ on the element $K$, $\Omega_\mathcal{T}$ is the domain covered by the mesh $\mathcal{T}_h$, let us introduce the following finite element space
$$V_h^p=\left\{v\in H^1(\Omega_{\mathcal{T}}):v|_K\in\mathbb{P}_p(K)~~\forall K\in\mathcal{T}_h\right\}.$$
For simplicity we will write the $L^2$-norm on a domain $\Theta$, $\|\cdot\|_{L^2(\Theta)}$ as $\|\cdot\|_\Theta$.
The domain $\Omega$ is embedded in a mesh $\mathcal{T}_h$. Figure \ref{domain_mesh} gives an example of a simple configuration.
\begin{figure}[h!]
\begin{center}
\input{mesh_domain}
\end{center}
\caption{Domain $\Omega$ embedded in a background mesh $\mathcal{T}_h$.}
\label{domain_mesh}
\end{figure}
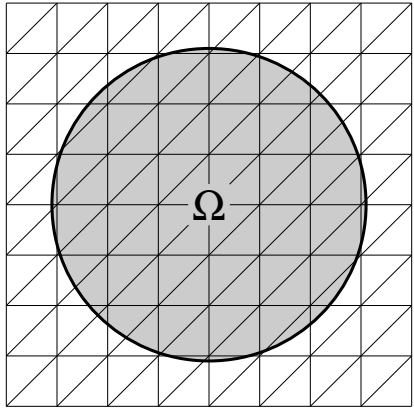
We define $\rho$ as the signed distance function, negative on the inside and positive on the outside of $\Omega$. The tubular neighbourhood of $\Gamma$ is defined as $U_\delta(\Gamma)=\{x\in\mathbb{R}^2:|\rho(x)|<\delta\}$. We consider a constant $\delta_0>0$ such that the closest point mapping $p(x):U_{\delta_0}(\Gamma)\rightarrow\Gamma$ is well defined and we have the identity $p(x)=x-\rho(x)n(p(x))$. We suppose that $\delta_0$ is chosen small enough such that $p(x)$ is a bijection.
Let the polygonal domain $\Omega_h$ with boundary $\Gamma_h$, be a
domain approximating $\Omega$. For simplicity we will assume that the
discrete domain is defined by the zero levelset of the nodal
interpolant of $\rho$ on $V_h^1$. Then on a triangle $K$ cut by the boundary, $\Gamma_h$
restricted to $K$ is a straight line. The discrete normal $n_h$
denotes the exterior unit normal to $\Gamma_h$. Observe that $n_h$ is constant on $\Gamma_h|_K$. We define $p_h(x,\varsigma)=x+\varsigma n_h(x)$, the function $\varrho_h$ is defined by $\varrho_h:\Gamma_h\longrightarrow\mathbb{R}$ such that $p_h(x,\varrho_h(x))\in \Gamma$ for all $x\in\Gamma_h$, for simplicity let $p_h(x)=p_h(x,\varrho_h(x))$. Observe that $\varrho_h$ is well defined for $h$ small enough (see \cite{Burman_2015_a}). We assume that $p_h(x,\varsigma)\in U_{\delta_0}(\Omega)$ for all $x\in \Gamma_h$ and all $\varsigma$ between $0$ and $\varrho_h(x)$. 
By our definition of $\Omega_h$ we have
\begin{equation}
\label{deltahdef}
\delta_h=\|\varrho_h\|_{L^{\infty}(\Gamma_h)}=\mathcal{O}(h^2).
\end{equation}
Let $\omega$ be a subset of $\Omega_{\mathcal{T}}$, we define
$$\mathcal{K}_h(\omega)=\left\{K\in\mathcal{T}_h~|~\overline{K}\cap\overline{\omega}\neq \emptyset\right\},\quad\quad\mathcal{N}_h(\omega)=\cup_{K\in\mathcal{K}_h(\omega)}K,$$
and the norm
$$\|v\|_{\mathcal{N}_h(\omega)}=\left(\sum_{K\in\mathcal{K}_h(\omega)}\|v\|_{K}^2\right)^\frac12.$$
We will use the notations
$$
\begin{aligned}
\mathcal{K}_h&=\mathcal{K}_h(\Omega\cup\Omega_h),\\
\mathcal{K}_\Gamma&=\mathcal{K}_h(\Gamma)\cup\mathcal{K}_h(\Gamma_h),\\
\mathcal{N}_h&=\mathcal{N}_h(\Omega\cup\Omega_h).
\end{aligned}
$$
We now recall the following trace inequalities for $v\in H^1(\mathcal{N}_h)$
\begin{align}
\label{trace}
\|v\|_{\partial K}&\lesssim (h_K^{-\frac12}\|v\|_{K}+h_K^{\frac12}\|\nabla v\|_{K})~~~~~\forall K\in\mathcal{K}_h,\\
\label{trace2}
\|v\|_{K\cap\Gamma}&\lesssim (h_K^{-\frac12}\|v\|_{K}+h_K^{\frac12}\|\nabla v\|_{K})~~~~~\forall K\in\mathcal{K}_h.
\end{align}
Inequality (\ref{trace2}) has been shown in \cite{Hansbo_2002_a}, we note that this is also true if we consider $\Gamma_h$ instead of $\Gamma$. For $v_h\in V_h^p$ the following inverse estimate holds
\begin{equation}
\label{inverse}
\|\nabla v_h\|_{K}\lesssim h_K^{-1}\|v_h\|_{K}~~~~~\forall K\in\mathcal{K}_h.
\end{equation}
The following inequality has been shown in \cite{Burman_2015_a} for all $v\in H^1(\Omega_h)$
\begin{equation}
\label{bound1}
\|v\|_{\Omega_h\backslash\Omega}^2\lesssim\delta_h^2\|\nabla v\cdot n\|_{\Omega_h\backslash\Omega}^2+\delta_h\|v\|_{\Gamma_h}^2.
\end{equation}
The ghost penalty \cite{Burman_2010_a} is introduced to ensure the well conditioning of the system matrix, it also provides the control of the gradient in case of small cut elements
$$J_h(u_h,v_h)=\gamma_g\sum_{F\in\mathcal{F}_G}\sum_{l=1}^{p}h^{2l-1}\langle \llbracket D_{n_F}^l u_h\rrbracket_F,\llbracket D_{n_F}^l v_h\rrbracket_F\rangle_F,$$
with $\mathcal{F}_G=\left\{F\in \mathcal{K}_\Gamma~|~F\cap(\Omega\cup\Omega_h)\neq \emptyset\right\},$ $\gamma_g$ the ghost penalty parameter and $n_F$ the unit normal to the face $F$ with fixed but arbitrary orientation. $D^l_{n_F}$ is the partial derivative of order $l$ in the direction $n_F$ and $\llbracket w \rrbracket_F=w_F^+-w_F^-$, with $w_F^{\pm}=\text{lim}_{s\rightarrow 0^+}w(x\mp sn_F)$ is the jump across a face $F$.
The following estimate has been shown in \cite{Massing_2014_a} for all $v_h\in V_h^p$
\begin{equation}
\label{lower_ghost}
\|\nabla v_h\|_{\mathcal{N}_h}^2\lesssim\|\nabla v_h\|_{\Omega_h}^2+J_h(v_h,v_h)\lesssim\|\nabla v_h\|_{\mathcal{N}_h}^2.
\end{equation}
We now construct an interpolation operator $\pi_h$. Let $\mathbb{E}$ be an $H^s$-extension on $\mathcal{N}_h$, $\mathbb{E} : H^s(\Omega)\rightarrow H^s(\mathcal{N}_h)$ such that for all $w\in$ $(\mathbb{E}w)|_{\Omega}=w$ and
\begin{equation}
\label{ineq_extension}
\|\mathbb{E} w\|_{H^s(\mathcal{N}_h)}\lesssim\|w\|_{H^s(\Omega)}\quad\quad\forall w\in H^s(\Omega),s\geq 0.
\end{equation}
For simplicity we will write $w$ instead of $\mathbb{E}w$. Let $\pi_h^* : H^s(\mathcal{N}_h)\rightarrow V_h^p$ be the Lagrange interpolant, we construct the interpolation operator $\pi_h$ such that
\begin{equation}
\label{def_interpol}
\pi_h u=\pi_h^*\mathbb{E} u.
\end{equation}
We have the interpolation estimate for $0\leq r \leq s \leq p+1$,
\begin{equation}
\label{interpolation_estimate_lagrange}
\|u-\pi_h^*u\|_{H^r(K)}\lesssim h^{s-r} \left|u\right|_{H^{s}(K)}~~~~~~~~~\forall K\in \mathcal{K}_h.
\end{equation}
Using the estimate (\ref{ineq_extension}) together with (\ref{interpolation_estimate_lagrange}) we have
\begin{equation}
\label{interpolation_estimate}
\|u-\pi_hu\|_{H^r(\mathcal{N}_h)}\lesssim h^{s-r} \left|u\right|_{H^{s}(\Omega)}.
\end{equation}
Let us introduce the norms
$$
\begin{aligned}
\vvvert w \vvvert_h^2&=\|\nabla w\|_{\Omega_h}^2+h^{-1}\|w\|_{\Gamma_h}^2+J_h(w,w),\\
\| w \|_*^2&=\vvvert w \vvvert_h^2+h\|\nabla w\cdot n_h\|_{\Gamma_h}^2+h^{-1}\|T_{1,k}(w)\|_{\Gamma_h}^2,
\end{aligned}
$$
with the Taylor expansion defined such that
\begin{equation}
\label{taylor_exp_definition}
T_{m,k}(u)(x)=\sum_{i=m}^{k}\frac{D_{n_h}^iu(x)}{i!}\varrho_h^i(x),
\end{equation}
$D_{n_h}^i$ is the derivative of order $i$ in the direction $n_h$. Using the estimate \eqref{interpolation_estimate} combined with the trace inequality \eqref{trace} it is straightforward to show
\begin{equation}
\label{estimate_new_norms}
\vvvert u-\pi_hu\vvvert_h\leq \| u-\pi_hu\|_*\lesssim h^p\left|u\right|_{H^{p+1}(\Omega)}.
\end{equation}

\section{Finite element formulation}
Here we use the boundary value correction approach from \cite{Burman_2015_a}, we write the extensions of $f$ and $u$ respectively as $f=\mathbb{E}f$ and $u=\mathbb{E}u$.
$$
\begin{aligned}
(f,v)_{\Omega_h}
&=(f+\Delta u, v)_{\Omega_h}-(\Delta u, v)_{\Omega_h}\\
&=(f+\Delta u, v)_{\Omega_h\backslash\Omega}+(\nabla u, \nabla v)_{\Omega_h}-\langle\nabla u \cdot n_h, v \rangle_{\Gamma_h}.
\end{aligned}
$$
We know that $f+\Delta u=0$ on $\Omega$. On $\Omega_h\backslash\Omega$ we have $f+\Delta u=\mathbb{E}f+\Delta\mathbb{E}u\neq 0$. Let us enforce weakly the boundary condition $u=g$ on $\Gamma$ by adding a consistent boundary term
$$
\begin{aligned}
(f,v)_{\Omega_h}=(f+\Delta u, v)_{\Omega_h\backslash\Omega}+(\nabla u, \nabla v)_{\Omega_h}-&\langle\nabla u \cdot n_h, v \rangle_{\Gamma_h}\\
&+\langle \nabla v \cdot n_h, u\circ p_h-g\circ p_h\rangle_{\Gamma_h}.
\end{aligned}
$$
Remark that this is equivalent to the penalty-free Nitsche's method \cite{Burman_2012_b,Boiveau_2015_a}. However, we cannot access $u\circ p_h$, so we use a Taylor approximation in the direction $n_h$ \eqref{taylor_exp_definition}
\begin{equation}
u\circ p_h(x)\approx T_{0,k}(u)(x).
\end{equation}
We note that in \eqref{taylor_exp_definition} we could replace $n_h$ by $n\circ p$ and $\varrho_h$ by $\varrho$ if these quantities were available (as mentioned in \cite{Burman_2015_a}).
Adding and subtracting the Taylor expansion in the Nitsche antisymmetric term and rearranging we obtain
\begin{equation}
\begin{split}
\label{weakequation}
(\nabla u, \nabla v)_{\Omega_h}&-\langle\nabla u \cdot n_h, v \rangle_{\Gamma_h}+\langle \nabla v \cdot n_h, T_{0,k}(u)\rangle_{\Gamma_h}+(f+\Delta u, v)_{\Omega_h\backslash\Omega}\\
&+\langle \nabla v \cdot n_h, u\circ p_h-T_{0,k}(u)\rangle_{\Gamma_h}=(f,v)_{\Omega_h}+\langle \nabla v \cdot n_h, g\circ p_h\rangle_{\Gamma_h}.
\end{split}
\end{equation}
The discrete formulation is obtained by dropping the terms $(f+\Delta u, v)_{\Omega_h\backslash\Omega}$ and $\langle \nabla v \cdot n_h, u\circ p_h-T_{0,k}(u)\rangle_{\Gamma_h}$. Find $u_h\in V_h^p$
\begin{equation}
\label{formulation}
A_h(u_h,v_h)+J_h(u_h,v_h)=L_h(v_h)~~~~~~\forall v_h \in V_h^p,
\end{equation}
with the linear forms
$$
\begin{aligned}
A_h(u_h,v_h)&=(\nabla u_h, \nabla v_h)_{\Omega_h}-\langle\nabla u_h \cdot n_h, v_h \rangle_{\Gamma_h}+\langle \nabla v_h \cdot n_h, T_{0,k}(u_h)\rangle_{\Gamma_h},\\
L_h(v_h)&=(f,v_h)_{\Omega_h}+\langle \nabla v_h \cdot n_h, g\circ p_h\rangle_{\Gamma_h}.
\end{aligned}
$$
Using the definition of the Taylor expansion, the bilinear form $A_h$ can be written as
$$
\begin{aligned}
A_h(u_h,v_h)=(\nabla u_h, \nabla v_h)_{\Omega_h}-\langle\nabla u_h \cdot n_h, v_h \rangle_{\Gamma_h}+&\langle \nabla v_h \cdot n_h, u_h\rangle_{\Gamma_h}\\
&+\langle \nabla v_h \cdot n_h, T_{1,k}(u_h)\rangle_{\Gamma_h}.
\end{aligned}
$$
The terms that has been dropped in the discrete formulation are defined as
\begin{equation*}
\label{defB}
B_h(u,v_h)=(f+\Delta u, v_h)_{\Omega_h\backslash\Omega}+\langle \nabla v_h \cdot n_h, u\circ p_h-T_{0,k}(u)\rangle_{\Gamma_h}\quad\forall v_h\in V_h^p.
\end{equation*}
In Section \ref{sec:infsup} we show an inf-sup condition for the discrete formulation \eqref{formulation}. In Section \ref{sec:bdvaluecorrection} the high order terms of $B_h(u,v_h)$ are treated. Section \ref{sec:errorestimate} presents the error estimates.

\section{Inf-sup condition on $\Omega_h$}
\label{sec:infsup}
We assume that $\Omega_h$ is defined by the zero level set of the
nodal interpolant $\mathcal{I}_h \rho$ of the distance function $\rho$. We also
assume that $h$ is small enough so that a band of elements in $\mathcal{K}_h(\Gamma_h)$ 
is in the tubular $U_{\delta_0}(\Gamma)$ and that in every $K \in
\mathcal{K}_h(\Gamma_h)$ there holds

\[
\|\rho - \mathcal{I}_h \rho\|_{L^\infty(K)} + h \|\nabla (\rho - \mathcal{I}_h
\rho)\|_{L^\infty(K)} \leq c_\rho h
\]
where the constant $c_\rho$ only depends on the regularity of the
interface
and, since $\rho$ is a distance function, for $x\in \mathcal{K}_h(\Gamma_h)$ we have
\begin{equation}
\label{bound_grad_dist_f}
1\leq |\nabla \rho(x) |\leq C_1,
\end{equation}
with $C_1>1$ a constant of order $1$. We now introduce boundary patches that will be useful for the upcoming
inf-sup analysis. 
Let us consider the set $\mathcal{K}_h(\Gamma_h)$ and split it into
$N_p$ smaller disjoint sets of elements $\mathcal{K}_j$ with
$j=1,\dots,N_p$, then we define 
\[
P_j = \mathcal{K}_j \cup\{K \in \mathcal{T}_h : K \cap \Omega_h \neq \emptyset,\exists K' \in \mathcal{K}_j \mbox{ such
  that } K \cap K' \ne \emptyset\}.
\]
This means that $P_j$ consists of the elements on $\mathcal{K}_j$ and
its neighbours that intersect $\Omega_h$ (we assume here that the mesh is truncated beyond
$\mathcal{K}_h(\Gamma_h)$ so that there are no exterior neighbours,
otherwise it is straightforward to handle them separately). Observe
that the patches $P_j$ overlap. For each patch $P_j$ we define the faces $F_j^1$ and $F_j^2$ where $\partial P_j\cap \Gamma_h\neq \emptyset$. We define the interior elements of the patch by 
\[
\mathcal{K}_j^\circ:= \{ K \in \mathcal{K}_h(\Gamma_h)\cap P_j: K \cap (F_j^1\cup F_j^2)= \emptyset \}.
\]
Let $I_{P_j}$ be the set of vertices $\{x_i\}$
in the patch $P_j$ and the cardinality of $I_{P_j}$ is $N_{P_j}$. We define the set of mesh vertices $I_j$ that are in the interior
 of the patch $P_j$ or on the outer boundary,
$$
I_j=\left\{x_i\in K: K \in \mathcal{K}_j^\circ\right\},
$$
Figure \ref{patch} shows an example of a patch. Let $\Gamma_j=\Gamma_h \cap \mathcal{K}_j$ denote the part of the boundary included in the patch $P_j$, for all $j$, the patch $P_j$ has the following properties
\begin{equation}
\label{order_patch}
\mbox{meas}_1(\Gamma_j)\sim h ~~~~\text{and}~~~~\mbox{meas}_2(P_j)\sim h^2.
\end{equation}
In \eqref{order_patch} we can control the constant in both relations
by choosing the patches to contain more elements (but uniformly under
refinement).

The function $v_\Gamma$ is defined such that $v_\Gamma=\sum_{j=1}^{N_p}v_j$, for each patch $P_j$, the function $v_j$ has the form $v_j=\zeta\tilde{\varphi}_j$ with $\zeta\in\mathbb{R}$. Let $\tilde{\varphi}_j\in V_h^1$ be defined for each node $x_i\in\mathcal{T}_h$ such that
\begin{displaymath}
\tilde{\varphi}_j(x_i)=
\left\{
\begin{array}{rllr}
0&\text{for}& x_i \in I_{P_j} \setminus I_j\\
-\rho(x_i) &\text{for}& x_i \in I_j,
\end{array}\right.
\end{displaymath}
with $i=1,\dots,N_{P_j}$.
\begin{figure}[h!]
\begin{center}
\input{patch_structured}
\end{center}
\caption{Example of a patch $P_j$, in this case $K_3\cup K_4\cup K_5=\mathcal{K}_j=\mathcal{K}_j^\circ$, $\tilde{\varphi}_j$ is equal to zero on the nonfilled nodes.}
\label{patch}
\end{figure}
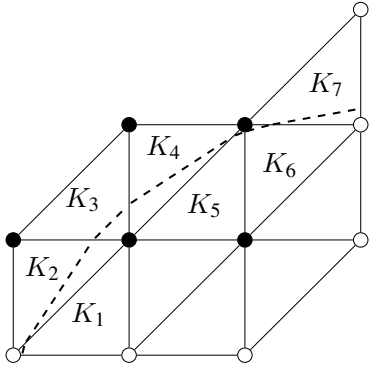
By the Poincar\'e inequality on a patch $P_j$ the following inequality holds
\begin{equation}
\label{poincare}
\|v_j\|_{P_j}\lesssim h\|\nabla v_j\|_{P_j}.
\end{equation}
\begin{lemma}
\label{keyineq_lemma}
For every patch $P_j$ with $1\leq j \leq N_p$ ; $\forall r_j\in\mathbb{R}$ there exists $v_j\in V_h^p$ such that
\begin{equation}
\label{keyineq_assump}
\textup{meas}(\Gamma_j)^{-1}\int_{\Gamma_j} \nabla v_j\cdot n_h ~\textup{d}s=r_j,
\end{equation}
and the following property holds
\begin{equation}
\label{keyineq_ineq}
\|\nabla v_j\|_{P_j}\lesssim\|h^\frac12 r_j\|_{\Gamma_j}.
\end{equation}
\end{lemma}
\begin{proof}
The functions $v_j$ and $\tilde{\varphi}_j$ are defined as previously described, let
$$\Xi_j=\text{meas}(\Gamma_j)^{-1}\int_{\Gamma_j}\nabla
\tilde{\varphi}_j\cdot n_h~\text{d}s.$$
We will first prove that for shape regular meshes, $\Xi_j$ is strictly
negative and bounded away from zero uniformly in $h$, provided the hidden constant of
the upper bound in \eqref{order_patch} is chosen large enough.
Let $K_1,\dots,K_m$ be the triangles crossed by $\Gamma_h$ within a
patch $P_j$, the numbering of the crossed elements is done in order
following a path of $\Gamma_j$ as in figure \ref{patch}. The assumption
\eqref{order_patch} says that the number of these triangles should be uniformly bounded
by some $m$. We will show that the upper bound $m$ only depends on the
shape regularity of the mesh. The mesh size on the other hand must be
small enough to resolve the boundary. The triangles $K_1$, $K_2$, $K_{m-1}$ and $K_m$ will have
nodes on the boundary of $\partial P_i$ where $\tilde \varphi=0$ (on $F_j^1$ or $F_j^2$). Let us merge $K_1$ and $K_2$ (resp. $K_{m-1}$ and $K_m$) into one quadrilateral element $K_{1}^\square$ (resp. $K_{m}^\square$). Observe that for the triangles $K_i$, $i=3,...,m-2$
there holds $K_i \in \mathcal{K}_j^\circ$ and with $\mathcal{I}_h$ denoting the standard nodal interpolant on
piecewise linear functions,
\[
(n_h \cdot \nabla  \tilde{\varphi}_j\vert_{K_i}) = |\nabla
\mathcal{I}_h \rho \vert_{K_i}|.
\]
On $K_{1}^\square$ and $K_{m}^\square$ the following upper bound holds
trivially
\[
(n_h \cdot \nabla  \tilde{\varphi}_j\vert_{K_i^\square}) < |\nabla
\tilde{\varphi}_j\vert_{K_i^\square}|, \quad i=1,m.
\]
Consider now 
\[
\int_{\Gamma_h \cap P_j} \nabla  \tilde{\varphi}_j\cdot
n_h~\text{d}s  = \sum_{i=1}^m
\int_{\Gamma_j\cap K_i} \nabla  \tilde{\varphi}_j\cdot
n_h~\text{d}s =  T_1^\square+T_m^\square+\sum_{i=3}^{m-2} T_i.
\]
For $i \in \{3,\hdots,m-2\}$ using \eqref{bound_grad_dist_f} we have
\begin{multline*}
T_i = |\nabla
\pi_h \rho \vert_{K_i}| \mbox{meas}(\Gamma_j\cap K_i) \ge
\mbox{meas}(\Gamma_j\cap K_i) (1- |\nabla
(1-\mathcal{I}_h) \rho \vert_{K_i}|) \\
\ge (1 - c_{\rho} h)  \mbox{meas}(\Gamma_j\cap K_i).
\end{multline*}
For $i \in \{1,m\}$ we have
\begin{multline*}
T_i^\square \leq h^{\frac12} \|\nabla \tilde \varphi_j\|_{F_j \cap K_i^\square} \leq
\|\nabla \tilde \varphi_j\|_{K_i^\square} \leq C_i h^{-1} \|\tilde
\varphi_j\|_{K_i^\square} 
\leq C_i h^{-1/2} \|\mathcal{I}_h \rho\|_{\partial K_i^\square \cap \partial
  \mathcal{K}_j^\circ}  \\\leq  C_i h^{-1/2} (\|\rho\|_{\partial K_i^\square \cap \partial
  \mathcal{K}_j^\circ} +\|\rho - \mathcal{I}_h \rho\|_{\partial K_i^\square \cap \partial
  \mathcal{K}_j^\circ} )
\leq c_\partial h.
\end{multline*}
The right hand side of these two inequalities depend only on 
the shape regularity of the mesh. Hence, using that $\mbox{meas}(\Gamma \cap
\mathcal{K}_j^\circ)\ge \mbox{meas}(\Gamma_j) -2h$
\begin{multline*}
\mbox{meas}(\Gamma_j)^{-1} \sum_{i=1}^m
\int_{\Gamma_j\cap K_i} \nabla  \tilde{\varphi}_j\cdot
n_h~\text{d}s\\ \ge  (1 - c_{\rho} h)  (1-2h \mbox{meas}(\Gamma_j)^{-1} )- 2  c_\partial h \mbox{meas}(\Gamma_j)^{-1}.
\end{multline*}
Assume now that $h$ is sufficiently small so that $ (1 - c_{\rho}
  h)>\frac12$ then
\[
\mbox{meas}(\Gamma_j)^{-1} \sum_{i=1}^m
\int_{\Gamma_j\cap K_i}\nabla  \tilde{\varphi}_j\cdot
n_h~\text{d}s \ge \frac12 - (2 c_\partial +1) h \mbox{meas}(\Gamma_j)^{-1}.
\]
If the lower constant of the left relation of \eqref{order_patch} is
larger than $4 (2 c_\partial +1)$ we may conclude that
\[
\Xi_j =\mbox{meas}(\Gamma_j)^{-1} \sum_{i=1}^m
\int_{\Gamma_j\cap K_i}\nabla  \tilde{\varphi}_j\cdot
n_h~\text{d}s \ge \frac14,
\]
which shows the uniform lower bound.
Note that the constant $c_\partial$ only depends on the
curvature of the boundary and the mesh geometry.
Thanks to this lower bound we may define the normalised function
$\varphi_j$ by
$$
\varphi_j=\Xi_j^{-1}\tilde{\varphi}_j.
$$
By definition there holds
\begin{equation}
\label{keyineq_proof_1}
\textup{meas}(\Gamma_j)^{-1}\int_{\Gamma_j}\nabla\varphi_j\cdot n_h~\textup{d}s=1,
\end{equation}
and
\begin{equation}
\|\nabla \varphi_j\|_{P_j} = \Xi_j^{-1} \|\nabla \tilde \varphi_j\|_{P_j}.
\end{equation}
The right hand side can be bounded as follows
\[
 \|\nabla \tilde \varphi_j\|_{P_j} \leq  \|\nabla \tilde
 \varphi_j\|_{\mathcal{K}_j^\circ} + \|\nabla \tilde
 \varphi_j\|_{P_j \setminus \mathcal{K}_j^\circ} = T_1 + T_2,
\]
\[
T_1 \leq  \|\nabla  (\pi_h \rho - \rho)\|_{\mathcal{K}_j^\circ} +
\|\nabla \rho \|_{\mathcal{K}_j^\circ} \leq C h^{\frac12} \mbox{meas}(\Gamma_j)^{\frac12}, 
\]
\begin{multline*}
T_2 \leq C_i h^{-1} \| \tilde
 \varphi_j\|_{P_i \setminus \mathcal{K}_j^\circ} \leq   C h^{-1/2} \|
  \mathcal{I}_h \rho\|_{\partial \mathcal{K}_j^\circ} \\
\leq C h^{-1/2}( \|
 \rho\|_{\partial \mathcal{K}_j^\circ}+ \|
 \mathcal{I}_h \rho - \rho\|_{\partial \mathcal{K}_j^\circ})
\leq C h^{\frac12} \mbox{meas}(\Gamma_j)^{\frac12}.
\end{multline*}
We conclude that
\begin{equation}
\label{keyineq_proof_2}
\|\nabla \varphi_j\|_{P_j} \leq  C h^{\frac12} \mbox{meas}(\Gamma_j).
\end{equation}
Let $v_j=r_j\varphi_j$, then condition (\ref{keyineq_assump}) is verified considering (\ref{keyineq_proof_1}). The upper bound (\ref{keyineq_ineq}) is obtained using (\ref{keyineq_proof_2}), (\ref{order_patch}) and
$$\|\nabla v_j\|_{P_j}= |r_j|\|\nabla \varphi_j\|_{P_j}\lesssim\text{meas}(\Gamma_j)^\frac12|r_j|h^\frac12=\|h^\frac12 r_j\|_{\Gamma_j}.$$
 \end{proof}
\begin{lemma}
\label{lowertriple}
For $u_h,v_h\in V_h^p$ with $v_h=u_h+\alpha v_\Gamma$, there exists a positive constant $\beta_0$ such that the following inequality holds
$$\beta_0\vvvert u_h \vvvert_{h}^2\leq A_h(u_h,v_h) + J_h(u_h,v_h).$$
\end{lemma}
\begin{proof}
Using \eqref{lower_ghost} it is straightforward to obtain
\begin{equation*}
\begin{split}
(A_h+J_h)(u_h,u_h)=&~(\nabla u_h,\nabla u_h)_{\Omega_h}+\langle \nabla u_h \cdot n_h, T_{1,k}(u_h)\rangle_{\Gamma_h}+J_h(u_h,u_h)\\
&\gtrsim\|\nabla u_h\|_{\mathcal{N}_h}^2+\langle \nabla u_h \cdot n_h, T_{1,k}(u_h)\rangle_{\Gamma_h}.
\end{split}
\end{equation*}
We bound the second term using the trace inequality \eqref{trace2}, inverse inequality \eqref{inverse} and \eqref{deltahdef}
$$
\begin{aligned}
\langle \nabla u_h \cdot n_h, T_{1,k}(u_h)\rangle_{\Gamma_h}&\leq\|\nabla u_h \cdot n_h\|_{\Gamma_h}\|T_{1,k}(u_h)\|_{\Gamma_h}\\
&\lesssim h^{-1}\|\nabla u_h\|_{\mathcal{N}_h(\Gamma_h)}\|T_{1,k}(u_h)\|_{\mathcal{N}_h(\Gamma_h)}\\
&\lesssim\gamma(h)\|\nabla u_h\|_{\mathcal{N}_h(\Gamma_h)}^2,
\end{aligned}
$$
where $\gamma(h)=\sum_{i=1}^{k}\left(\frac{\delta_h}{h}\right)^i$, note that $\mathcal{O}(\gamma(h))=h$. Considering $v_\Gamma$ as defined in Section \ref{preliminaries}, for the $j^{th}$ patch $P_j$ we get
$$
\begin{aligned}
\alpha(A_h+J_h)(u_h,v_j)=\alpha&(\nabla u_h,\nabla v_j)_{P_j\cap \Omega_h}-\alpha\langle\nabla u_h \cdot n_h, v_j\rangle_{\Gamma_j}+\alpha J_h(u_h,v_j)\\
&+\alpha\langle\nabla v_j \cdot n_h, u_h\rangle_{\Gamma_j}+\alpha\langle \nabla v_j \cdot n_h, T_{1,k}(u_h)\rangle_{\Gamma_j},
\end{aligned}
$$
 applying Cauchy-Schwarz inequality together with \eqref{lower_ghost} we can write
$$
\begin{aligned}
\alpha(\nabla u_h&,\nabla v_j)_{P_j\cap \Omega_h}+\alpha J_h(u_h,v_j)\\
&\leq\alpha\|\nabla u_h\|_{P_j\cap \Omega_h}\|\nabla v_j\|_{P_j\cap \Omega_h}+\alpha J_h(u_h,u_h)^\frac12J_h(v_j,v_j)^\frac12,\\
&\leq\epsilon\| \nabla u_h\|_{P_j}^2+\frac{C\alpha^2}{4\epsilon}\| \nabla v_j\|_{P_j}^2.
\end{aligned}
$$
Using the trace inequality (\ref{trace2}), the inverse inequality (\ref{inverse}) and the inequality (\ref{poincare}) we can write the following
$$
\begin{aligned}
\alpha\langle\nabla u_h \cdot n_h, v_j\rangle_{\Gamma_j}
&\leq\alpha\|  \nabla u_h\cdot n_h\|_{\Gamma_j}\| v_j\|_{\Gamma_j}\lesssim\alpha h^{-1}\|  \nabla u_h\|_{P_j}\| v_j\|_{P_j}\\
&\lesssim\alpha\|  \nabla u_h\|_{P_j}\| \nabla v_j\|_{P_j}\leq\epsilon\| \nabla u_h\|_{P_j}^2+\frac{C\alpha^2}{4\epsilon}\| \nabla v_j\|_{P_j}^2.
\end{aligned}
$$
Let us consider the average $\overline{u}^j=\text{meas}(\Gamma_j)^{-1}\int_{\Gamma_j}u_h~\text{d}s$. Using Lemma \ref{keyineq_lemma} and choosing $r_j=h^{-1}\overline{u}^j$ we get the inequality
\begin{equation}
\label{keyineq}
\|\nabla v_j\|_{P_j}\lesssim \|h^{-\frac12}\overline{u}^j\|_{\Gamma_j}.
\end{equation}
Our choice of $r_j$ allows us to write
$$
\alpha\langle\nabla v_j \cdot n_h, u_h\rangle_{\Gamma_j}=\alpha\|h^{-\frac12}\overline{u}^j\|_{\Gamma_j}^2+\alpha\langle\nabla v_j \cdot n_h, u_h-\overline{u}^j\rangle_{\Gamma_j}.
$$
It is straightforward to observe
\begin{equation}
\label{stdapprox}
\|u_h-\overline{u}^j\|_{\Gamma_j}\leq Ch \|\nabla u_h \|_{\Gamma_j},
\end{equation}
combining this result with the trace and inverse inequalities, we can show
$$
\alpha\langle\nabla v_j \cdot n_h, u_h-\overline{u}^j\rangle_{\Gamma_j}\leq\epsilon\| \nabla u_h\|_{P_j}^2+\frac{C\alpha^2}{4\epsilon}\| \nabla v_j\|_{P_j}^2.
$$
Using \eqref{deltahdef} once again
$$
\alpha\langle \nabla v_j \cdot n_h, T_{1,k}(u_h)\rangle_{\Gamma_j}\leq\epsilon\gamma(h)^2\|\nabla u_h\|_{P_j}^2+\frac{C\alpha^2}{4\epsilon}\|\nabla v_j\|_{P_j}^2.
$$
Each term has been bounded, we can now get back to
\begin{equation*}
\begin{split}
(A_h+J_h)(u_h,v_h)\geq (C-\gamma(h))&\| \nabla u_h \|_{\mathcal{N}_h}^2+\alpha\sum_{j=1}^{N_p}\|h^{-\frac12}\overline{u}^j\|_{\Gamma_j}^2\\
&-4\epsilon\sum_{j=1}^{N_p}\| \nabla u_h\|_{P_j}^2-\frac{C\alpha^2}{\epsilon}\sum_{j=1}^{N_p}\| \nabla v_j\|_{P_j}^2,
\end{split}
\end{equation*}
using (\ref{keyineq}) and rearranging, we obtain
\begin{equation*}
\begin{split}
(A_h+J_h)(u_h,v_h)\geq (C-\gamma(h)-4\epsilon)\| \nabla u_h \|_{\mathcal{N}_h}^2+\alpha\Big(1-\frac{C\alpha}{\epsilon}\Big)\sum_{j=1}^{N_p}\|h^{-\frac12}\overline{u}^j\|_{\Gamma_j}^2.
\end{split}
\end{equation*}
Using (\ref{stdapprox}), the trace inequality and the inverse inequality we can show
\begin{equation*}
\|h^{-\frac12}\overline{u}^j\|_{\Gamma_j}^2\geq\|h^{-\frac12} u_h\|_{\Gamma_j}^2-C'\|\nabla u_h\|_{P_j}^2,
\end{equation*}
using this result together with \eqref{lower_ghost} we obtain
\begin{equation*}
\begin{split}
(A_h+J_h)(u_h,v_h)\geq&~(C-\gamma(h)-4\epsilon-C'\alpha)\left(\| \nabla u_h \|_{\Omega_h}^2+J_h(u_h,u_h)\right)\\&+\alpha\Big(1-\frac{C\alpha}{\epsilon}\Big)\sum_{j=1}^{N_p}\|h^{-\frac12}u_h\|_{\Gamma_j}^2.
\end{split}
\end{equation*}
It is easy to choose $\epsilon$ and $\alpha$ such that the two terms of this expression are positive, for example, by choosing $\epsilon=\frac1{16}$ we obtain $\alpha=\text{min}(\frac{C-\gamma(h)-\frac14}{C'},\frac{1}{16C})$.
 \end{proof}
\begin{theorem}
\label{infsup}
There exists a positive constant $\beta$ such that for all function $u_h\in V_h^p$ the following inequality holds
$$\beta\vvvert u_h \vvvert_{h}\leq\underset{v_h\in V_h^p}{\textup{sup}}\frac{A_h(u_h,v_h) + J_h(u_h,v_h)}{\vvvert v_h \vvvert_{h}}.$$
\end{theorem}
\begin{proof}
Considering Lemma \ref{lowertriple} the only thing to show is $\vvvert v_h \vvvert_{h} \lesssim \vvvert u_h \vvvert_{h}$,
$$\vvvert v_h \vvvert_{h} \leq\vvvert u_h \vvvert_{h} + \vvvert v_\Gamma \vvvert_{h}~~~\text{with}~~~~~\vvvert v_\Gamma \vvvert_{h}\leq\sum_{j=1}^{N_p}\vvvert v_j \vvvert_{h},$$
$$\vvvert v_j \vvvert_{h}^2=\| \nabla v_j\|_{P_j\cap\Omega_h}^2+\| h^{-\frac12} v_j\|_{\Gamma_j}^2+J_h(v_j,v_j).$$
Using the trace inequality together with (\ref{poincare}) and (\ref{keyineq}) we observe
$$\| h^{-\frac12} v_j\|_{\Gamma_j}^2\lesssim\| \nabla v_j\|_{P_j}^2\lesssim\|h^{-\frac12}\overline{u}^j\|_{\Gamma_j}\lesssim\|h^{-\frac12}u_h\|_{\Gamma_j}.$$
We conclude using \eqref{lower_ghost}.
 \end{proof}

\section{Boundary value correction}
\label{sec:bdvaluecorrection}
The goal of this section is bound the two high order terms that has been dropped in the finite element formulation \eqref{formulation}.
\begin{theorem}
Let $B_h$ be the bilinear form as defined in \eqref{defB} the following holds $\forall v_h\in V_h^p$
\label{bdcorrection}
$$B_h(u,v_h)\lesssim (h^{-\frac12}\delta_h^{k+1}\underset{0\leq t\leq\delta_0}{\textup{sup}}\|D^{k+1}u\|_{\Gamma_t}+ \delta_h^{l+1} \underset{0\leq t\leq\delta_0}{\textup{sup}} \|D_n^l(f+\Delta u)\|_{\Gamma_t})\vvvert v_h\vvvert_{h}.$$
\end{theorem}
\begin{proof}
Using the Cauchy-Schwarz inequality
\begin{equation*}
|B_h(u,v_h)|\leq\|u\circ p_h(x)-T_{0,k}(u)(x)\|_{\Gamma_h}\|\nabla v_h \cdot n_h\|_{\Gamma_h}+\|f+\Delta u\|_{\Omega_h\backslash \Omega}\| v_h\|_{\Omega_h\backslash \Omega}.
\end{equation*}
By definition of the Taylor approximation we have
\begin{equation*}
|u\circ p_h(x)-T_{0,k}(u)(x)|=\Big|\int_0^{\varrho_h(x)}D_{n_h}^{k+1} u(x(s))(\varrho_h(x)-s)^k~\text{d}s\Big|.
\end{equation*}
Using the Cauchy Schwarz inequality
$$
\begin{aligned}
\|u\circ p_h(x)-T_{0,k}(u)(x)\|_{\Gamma_h}^2
&\leq\int_{\Gamma_h}\|D_{n_h}^{k+1}u\|_{I_x}^2\|(\varrho_h(x)-s)^k\|_{I_x}^2~\text{d}s\\
&\leq\int_{\Gamma_h}\|D_{n_h}^{k+1}u\|_{I_{\delta_h}}^2|\varrho_h(x)|^{2k+1}~\text{d}s\\
&\leq\delta_h^{2k+1}\|D^{k+1}u\|_{U_{\delta_h}(\Gamma_h)}^2\\
&\leq\delta_h^{2k+2}\underset{0\leq t\leq\delta_0}{\text{sup}}\|D^{k+1}u\|_{\Gamma_t}^2,
\end{aligned}
$$
where $\Gamma_t=\{ x\in \Omega~:~|\rho(x)|=t\}$ is the levelset with distance $t$ to the boundary $\Gamma$, following the approach from \cite{Burman_2015_a}. Suppose that
$$f+\Delta u \in H^l(U_{\delta_0}(\Omega)),$$
this property holds if $f\in H^l(\Omega)$ by applying \eqref{regularity} and \eqref{ineq_extension}. Using $\Omega_h\backslash \Omega\in U_\delta(\Gamma)$ and $\delta\sim \delta_h$ it follows that
\begin{equation*}
\|f+\Delta u\|_{\Omega_h\backslash \Omega}
\lesssim \delta_h^l \|D_n^l(f+\Delta u)\|_{\Omega_h\backslash \Omega}\lesssim \delta_h^{l+\frac12} \underset{0\leq t\leq\delta_0}{\text{sup}} \|D_n^l(f+\Delta u)\|_{\Gamma_t},
\end{equation*}
Then we can bound the bilinear form $B_h$, using $\|v\|_{\Omega_h\backslash \Omega}\lesssim\delta_h^{1/2}
\vvvert v \vvvert_h$ (deduced from \eqref{bound1}) the trace and inverse inequalities $\forall v_h\in V_h^p$
\begin{equation}
\begin{split}
\label{B_for_L2}
|&B_h(u,v_h)|\\
&\lesssim\delta_h^{k+1}\underset{0\leq t\leq\delta_0}{\text{sup}}\|D^{k+1}u\|_{\Gamma_t}\|\nabla v_h \cdot n_h\|_{\Gamma_h}+ \delta_h^{l+\frac12} \underset{0\leq t\leq\delta_0}{\text{sup}} \|D_n^l(f+\Delta u)\|_{\Gamma_t}\| v_h\|_{\Omega_h\backslash \Omega}\\
&\lesssim(h^{-\frac12}\delta_h^{k+1}\underset{0\leq t\leq\delta_0}{\text{sup}}\|D^{k+1}u\|_{\Gamma_t}+ \delta_h^{l+1} \underset{0\leq t\leq\delta_0}{\text{sup}} \|D_n^l(f+\Delta u)\|_{\Gamma_t})\vvvert v_h\vvvert_{h}.
\end{split}
\end{equation}
 \end{proof}

\section{A priori error estimate}
\label{sec:errorestimate}
The formulation (\ref{formulation}) satisfies the following consistency relation (Galerkin orthogonality).
\begin{lemma}
\label{galerkin}
Let $u_h\in V_h^p$ be the solution of (\ref{formulation}) and $u\in H^2(\mathcal{N}_h)$ be the solution of (\ref{poisson}), then
$$A_h(u-u_h,v_h)-J_h(u_h,v_h)+B_h(u,v_h)=0~~~,~~~\forall v_h\in V_h^p.$$ 
\end{lemma}
\begin{proof}
Subtracting \eqref{weakequation} and \eqref{formulation} this is straightforward. 
\end{proof}
\begin{lemma}
\label{triplestar}
Let $w\in H^2(\mathcal{N}_h)+ V_h^p$ and $v_h\in V_h^p$ there exists a positive constant $M$ such that
$$A_h(w,v_h)\leq M \| w \|_*\vvvert v_h \vvvert_{h}.$$
\end{lemma}
\begin{proof}
Using the Cauchy-Schwarz inequality, trace inequality and inverse inequality we have
$$
\begin{aligned}
(\nabla w, \nabla v_h)_{\Omega_h}-\langle\nabla w \cdot n_h&, v_h \rangle_{\Gamma_h}+\langle \nabla v_h \cdot n_h, T_{0,k}(w)\rangle_{\Gamma_h}\\
&\lesssim \|\nabla w\|_{\Omega_h}\|\nabla v_h\|_{\Omega_h}+\|\nabla w \cdot n_h\|_{\Gamma_h}\| v_h\|_{\Gamma_h}\\
&\quad\quad\quad\quad+\| \nabla v_h \|_{\Omega_h}h^{-1}(\|w\|_{\Gamma_h}+\|T_{1,k}(w)\|_{\Gamma_h}).
\end{aligned}
$$
\end{proof}
Using the inf-sup condition from Section \ref{sec:infsup} and the estimate obtained in Section \ref{sec:bdvaluecorrection} the error in the triple norm can now be estimated.
\begin{proposition}
\label{cvgH1}
Let $u\in H^2(\mathcal{N}_h)$ be the solution of (\ref{poisson}) and $u_h\in V_h^p$ the solution of (\ref{formulation}), then
$$
\begin{aligned}
\vvvert u-u_h \vvvert_h\lesssim h^p\|u\|_{H^{p+1}(\Omega)}+h^{-\frac12}\delta_h^{k+1}&\underset{0\leq t\leq\delta_0}{\textup{sup}}\|D^{k+1}u\|_{\Gamma_t}\\
&+\delta_h^{l+1} \underset{0\leq t\leq\delta_0}{\textup{sup}} \|D_n^l(f+\Delta u)\|_{\Gamma_t}.
\end{aligned}
$$
\begin{table}[H]
\begin{center}
 \begin{tabular}{|c|c||c|c||c|c|} 
 \hline
 $p$ & $h^{p}$ & $k$ & $h^{-\frac12}\delta_h^{k+1}$ & $l$ &$\delta_h^{l+1}$ \\ [0.5ex] 
 \hline
  1& $h^1$ & 0 & $h^{1.5}$ & 0 & $h^2$ \\ 
 
 2 & $h^2$ & 1 & $h^{3.5}$ & 1 & $h^4$\\
 
 3 & $h^3$ & 2 & $h^{5.5}$ & 2 & $h^6$\\
 
  4& $h^4$ & 3 & $h^{7.5}$ & 3 & $h^8$\\
 \hline
\end{tabular}
\caption{Order of the terms in the estimation of the $H^1$-error depending on $p$, $k$ and $l$, assuming $\delta_h=\mathcal{O}(h^2)$. }
\end{center}
\end{table}
\end{proposition}
\begin{proof}
Using the Galerkin orthogonality of Lemma \ref{galerkin} we obtain
$$
\begin{aligned}
A_h(u_h&-\pi_h u,v_h)+J_h(u_h-\pi_h u,v_h)\\
&=A_h(u-\pi_h u,v_h)-J_h(\pi_h u,v_h)+(f+\Delta u, v_h)_{\Omega_h\backslash\Omega}\\
&\quad\quad\quad\quad\quad\quad\quad\quad\quad\quad\quad+\langle \nabla v_h \cdot n_h, u\circ p_h-T_{0,k}(u)\rangle_{\Gamma_h}.
\end{aligned}
$$
Using this result, Theorem \ref{infsup}, the Lemma \ref{triplestar} and $J_h(u, v_h)=0$ given by the regularity of $u$, we can write
$$
\begin{aligned}
\beta\vvvert u_h-&\pi_h u\vvvert_{h}
\leq\frac{A_h(u-\pi_h u,v_h)-J_h(\pi_h u,v_h)+B_h(u,v_h)}{\vvvert v_h\vvvert_{h}}\\
&\leq M\| u-\pi_hu\|_*+\frac{J_h(\pi_h u-u,\pi_h u-u)^\frac12J_h(v_h,v_h)^\frac12+B_h(u,v_h)}{\vvvert v_h\vvvert_{h}}\\
&\leq (M+1)\| u-\pi_hu\|_*+\frac{B_h(u,v_h)}{\vvvert v_h\vvvert_{h}}
\end{aligned}
$$
Applying the triangle inequality we can write
$$
\begin{aligned}
\vvvert u-u_h\vvvert_h&\leq\vvvert u-\pi_hu\vvvert_h+\vvvert u_h-\pi_hu\vvvert_{h}\\
&\leq\vvvert u-\pi_hu\vvvert_h+\frac{1}{\beta}\Big((M+1)\| u-\pi_hu\|_*+\frac{B_h(u,v_h)}{\vvvert v_h\vvvert_{h}}\Big),
\end{aligned}
$$
Using the estimate \eqref{estimate_new_norms} we conclude applying Theorem \ref{bdcorrection}.
 \end{proof}
The next proof requires these two inequalities and the assumption $\delta_h\lesssim h^2$
\begin{align}
\| T_{1,k}(u-u_h)\|_{\Gamma_h}&\lesssim h^{p+1}\|u\|_{H^{p+1}(\Omega)}+(h^{-\frac32}\delta_h)h\vvvert u-u_h\vvvert_h \label{bound2}\\
h^{\frac12}\|\nabla (u-u_h)\cdot n_h\|_{\Gamma_h}&\lesssim h^{p}\|u\|_{H^1(\Omega)}+\vvvert u-u_h\vvvert_h, \label{bound6}
\end{align}
inequality \eqref{bound2} is proved in \cite{Burman_2015_a}, \eqref{bound6} can be shown using the trace and inverse inequalities in the following way
$$
\begin{aligned}
h^\frac12\|\nabla (u-u_h)\cdot n_h\|_{\Gamma_h}
&\lesssim\|\nabla (u-u_h)\cdot n_h\|_{\mathcal{N}(\Gamma_h)}+h\|D^2 (u-u_h)\|_{\mathcal{N}(\Gamma_h)}\\
&\lesssim\|\nabla (u-\pi_hu)\cdot n_h\|_{\mathcal{N}(\Gamma_h)}+h\|D^2 (u-\pi_hu)\|_{\mathcal{N}(\Gamma_h)}\\
&\quad+\|\nabla (u_h-\pi_hu)\cdot n_h\|_{\mathcal{N}(\Gamma_h)}+h\|D^2 (u_h-\pi_hu)\|_{\mathcal{N}(\Gamma_h)}\\
&\lesssim h^p\|u\|_{H^1(\Omega)}+\|\nabla (u_h-\pi_hu)\|_{\mathcal{N}(\Gamma_h)}\\
&\lesssim h^p\|u\|_{H^1(\Omega)}+\|\nabla (u-u_h)\|_{\mathcal{N}(\Gamma_h)}.
\end{aligned}
$$

\begin{theorem}
\label{L2_proof_poisson}
Let $u\in H^2(\mathcal{N}_h)$ be the solution of (\ref{poisson}) and $u_h\in V_h^p$ the solution of (\ref{formulation}), we assume $\delta_h\lesssim h^2$ then
$$
\begin{aligned}
\| u-u_h \|_{\Omega_h}\lesssim h^{p+\frac12}\|u\|_{H^{p+1}(\Omega)}+\delta_h^{k+1}&\underset{0\leq t\leq\delta_0}{\textup{sup}}\|D^{k+1}u\|_{\Gamma_t}\\
&+ h^{\frac12}\delta_h^{l+1} \underset{0\leq t\leq\delta_0}{\textup{sup}} \|D_n^l(f+\Delta u)\|_{\Gamma_t}.
\end{aligned}
$$
\begin{table}[H]
\begin{center}
 \begin{tabular}{|c|c||c|c||c|c|} 
 \hline
 $p$ & $h^{p+\frac12}$ & $k$ & $\delta_h^{k+1}$ & $l$ &$h^\frac12\delta_h^{l+1}$ \\ [0.5ex] 
 \hline
  1& $h^{1.5}$ & 0 & $h^{2}$ & 0 & $h^{2.5}$ \\ 
 
 2 & $h^{2.5}$ & 1 & $h^{4}$ & 1 & $h^{4.5}$\\
 
 3 & $h^{3.5}$ & 2 & $h^{6}$ & 2 & $h^{6.5}$\\
 
  4& $h^{4.5}$ & 3 & $h^{8}$ & 3 & $h^{8.5}$\\
 \hline
\end{tabular}
\caption{Order of the terms in the estimation of the $L^2$-error depending on $p$, $k$ and $l$, assuming $\delta_h=\mathcal{O}(h^2)$. }
\end{center}
\end{table}
\end{theorem}
\begin{proof}
We define the function $\psi$ such that
\begin{displaymath}
\psi=
\left\{
\begin{array}{rllr}
u-u_h&\text{in}& \Omega_h\\
0&\text{in}& \Omega\backslash\Omega_h.
\end{array}\right.
\end{displaymath}
Let $z$ satisfy the adjoint problem
\begin{equation}
\begin{aligned}
-\Delta z&=\psi\quad\text{in}~\Omega,\nonumber\\
z&=0\quad~\text{on}~\Gamma,
\end{aligned}
\end{equation}
$z$ is extended to $U_{\delta_0}(\Omega)$ using the extension operator. In this framework, the following estimates hold
\begin{align}
\|z\|_{H^2(\Omega)}&\lesssim\|\psi\|_{\Omega\cap\Omega_h}\label{bound3},\\
\|z\|_{\Omega_h\backslash \Omega}&\lesssim \delta_h \|\nabla z\cdot n\|_{U_{\delta_h}(\Gamma)}\label{bound4},\\
\|z\|_{\Gamma_h}&\lesssim \delta_h^{\frac12}\|\nabla z \cdot n\|_{U_{\delta_h}(\Gamma)}\label{bound5},
\end{align}
inequalities \eqref{bound4} and \eqref{bound5} has been shown in \cite{Burman_2015_a}. Using integration by parts, the $L^2$-error on $\Omega_h$ can be written as
$$
\begin{aligned}
\|u-&u_h\|_{\Omega_h}
=(u-u_h,\psi+\Delta z)_{\Omega_h}-(u-u_h,\Delta z)_{\Omega_h}\\
&=(u-u_h,\psi+\Delta z)_{\Omega_h\backslash\Omega}+(\nabla (u-u_h),\nabla z)_{\Omega_h}-\langle u-u_h,\nabla z \cdot n_h \rangle_{\Gamma_h}\\
&=(u-u_h,\psi+\Delta z)_{\Omega_h\backslash\Omega}+A_h(u-u_h,z)+\langle \nabla (u-u_h) \cdot n_h,z\rangle_{\Gamma_h}\\
&\quad\quad-2\langle\nabla z\cdot n_h,u-u_h\rangle_{\Gamma_h}-\langle\nabla z\cdot n_h,T_{1,k}(u-u_h)\rangle_{\Gamma_h}.
\end{aligned}
$$
Using \eqref{bound1}, the property of the extension operator \eqref{ineq_extension} and \eqref{bound3}, we can write
$$
\begin{aligned}
(u-u_h,&\psi+\Delta z)_{\Omega_h\backslash\Omega}
\leq\|u-u_h\|_{\Omega_h\backslash\Omega}\|\psi+\Delta z\|_{\Omega_h\backslash\Omega}\\
&\leq(\delta_h^2\|\nabla(u-u_h)\cdot n\|_{\Omega_h\backslash\Omega}^2+\delta_h\|u-u_h\|_{\Gamma_h}^2)^{\frac12}(\|\psi\|_{\Omega_h\backslash\Omega}+\|\Delta z\|_{\Omega_h\backslash\Omega})\\
&\lesssim(\delta_h^2+h\delta_h)^{\frac12}\vvvert u-u_h\vvvert_{h}(\|u-u_h\|_{\Omega_h\backslash\Omega}+\|z\|_{H^2(\Omega)})\\
&\lesssim(\delta_h^2+h\delta_h)^{\frac12}\vvvert u-u_h\vvvert_{h}\|u-u_h\|_{\Omega_h}.
\end{aligned}
$$
Using the interpolant defined by \eqref{def_interpol} we obtain
\begin{equation*}
A_h(u-u_h,z)\leq|A_h(u-u_h,z-\pi_hz)+A_h(u-u_h,\pi_hz)|,
\end{equation*}
by Cauchy-Schwarz inequality, \eqref{bound2}, \eqref{bound6} and \eqref{bound3} we have
$$
\begin{aligned}
&A_h(u-u_h,z-\pi_hz)\\
&\lesssim(\vvvert u-u_h\vvvert_h+h^\frac12\|\nabla (u-u_h)\cdot n_h\|_{\Gamma_h}+h^{-\frac12}\|T_{1,k}(u-u_h)\|_{\Gamma_h})\vvvert z-\pi_hz \vvvert_*\\
&\lesssim((1+h^{-1}\delta_h)h\vvvert u-u_h\vvvert_h+h^{p+1}\|u\|_{H^1(\Omega)})\| u-u_h\|_{\Omega_h}.
\end{aligned}
$$
The Galerkin orthogonality of Lemma \ref{galerkin} allows us to write
$$
\begin{aligned}
A_h(u-u_h,\pi_hz)
&\lesssim|B_h(u,\pi_hz)+J_h(u_h,\pi_hz)|\\
&\lesssim|B_h(u,\pi_hz)+J_h(u_h,u_h)^\frac12J_h(\pi_hz,\pi_hz)^\frac12|.
\end{aligned}
$$
From \cite{Burman_2012_a} and the properties of $z$ we have
$$J_h(\pi_hz,\pi_hz)^\frac12=J_h(\pi_hz-z,\pi_hz-z)^\frac12\lesssim h\|z\|_{H^{2}(\Omega)}\lesssim h\| u-u_h\|_{\Omega_h},$$
we also have the upper bound
\begin{equation*}
J_h(u_h,u_h)^\frac12\lesssim\vvvert u_h-\pi_hu\vvvert_h+J_h(\pi_hu,\pi_hu)^\frac12\lesssim \vvvert u_h-\pi_hu\vvvert_h+ h^p\|u\|_{H^{p+1}(\Omega)}.
\end{equation*}
Then using the proof of Proposition \ref{cvgH1} we have
$$
\begin{aligned}
|J_h(u,\pi_hz)|
\lesssim\Big(h^{p+1}\|u\|_{H^{p+1}(\Omega)}+&~h^{\frac12}\delta_h^{k+1}\underset{0\leq t\leq\delta_0}{\textup{sup}}\|D^{k+1}u\|_{\Gamma_t}\\&+h\delta_h^{l+1} \underset{0\leq t\leq\delta_0}{\textup{sup}}\|D_n^l(f+\Delta u)\|_{\Gamma_t}\Big)\| u-u_h\|_{\Omega_h}.
\end{aligned}
$$
Using equation \eqref{B_for_L2} the term $B_h(u,\pi_hz)$ can also be bounded with
$$
\begin{aligned}
|B_h(u,\pi_hz)|
\lesssim\delta_h^{k+1}\underset{0\leq t\leq\delta_0}{\text{sup}}&\|D^{k+1}u\|_{\Gamma_t}\|\nabla \pi_hz \cdot n_h\|_{\Gamma_h}\\
&+ \delta_h^{l+\frac12} \underset{0\leq t\leq\delta_0}{\text{sup}} \|D_n^l(f+\Delta u)\|_{\Gamma_t}\| \pi_hz\|_{\Omega_h\backslash \Omega}.
\end{aligned}
$$
Using the global trace inequality $\|\nabla z\cdot n_h\|_{\Gamma_h}\lesssim \|z\|_{H^2(\Omega_h)}$, \eqref{ineq_extension} and \eqref{bound3} we can write
$$
\begin{aligned}
\|\nabla \pi_hz \cdot n_h\|_{\Gamma_h}
&\leq\|\nabla (\pi_hz-z) \cdot n_h\|_{\Gamma_h}+\|\nabla z \cdot n_h\|_{\Gamma_h}\\
&\lesssim h^{-\frac12}\vvvert\pi_hz-z\vvvert_*+\|z\|_{H^2(\Omega_h)}\\
&\lesssim h^{\frac12}\|z\|_{H^2(\Omega)}+\|z\|_{H^2(\Omega_h)}\\
&\lesssim\|u-u_h\|_{\Omega_h}.
\end{aligned}
$$
Using inequalities \eqref{ineq_extension}, \eqref{bound3} and \eqref{bound4} we obtain
$$
\begin{aligned}
\|\pi_hz\|_{\Omega_h\backslash \Omega}
&\leq\|\pi_hz-z\|_{\Omega_h\backslash \Omega}+\|z\|_{\Omega_h\backslash \Omega}\\
&\leq h^2\|z\|_{H^2(\Omega)}+\delta_h \|\nabla z\|_{U_{\delta_h}(\Gamma)}\\
&\leq h^2\|u-u_h\|_{\Omega_h}.
\end{aligned}
$$
Then we obtain the upper bound
$$
\begin{aligned}
|B_h(u,\pi_hz)|
\lesssim\Big(&\delta_h^{k+1}\underset{0\leq t\leq\delta_0}{\text{sup}}\|D^{k+1}u\|_{\Gamma_t}\|\nabla \pi_hz \cdot n_h\|_{\Gamma_h}\\&+ h^2\delta_h^{l+\frac12} \underset{0\leq t\leq\delta_0}{\text{sup}} \|D_n^l(f+\Delta u)\|_{\Gamma_t}\| \pi_hz\|_{\Omega_h\backslash \Omega}\Big)\|u-u_h\|_{\Omega_h},
\end{aligned}
$$
and
$$
\begin{aligned}
A_h(u-u_h,z)\lesssim \Big(h^{p+1}&\|u\|_{H^{p+1}(\Omega)}+\delta_h^{k+1}\underset{0\leq t\leq\delta_0}{\text{sup}}\|D^{k+1}u\|_{\Gamma_t}\\&+h^2\delta_h^{l+\frac12} \underset{0\leq t\leq\delta_0}{\text{sup}} \|D_n^l(f+\Delta u)\|_{\Gamma_t}\Big)\| u-u_h\|_{\Omega_h}.
\end{aligned}
$$
Using \eqref{bound5} and \eqref{bound3} and the  we have
\begin{equation*}
\|z\|_{\Gamma_h}
\lesssim\delta_h^\frac12\|\nabla z\cdot n\|_{U_{\delta_h}(\Gamma)}\lesssim \delta_h\underset{0\leq t\leq\delta_0}{\text{sup}}\|\nabla z\cdot n\|_{\Gamma_t}\lesssim \delta_h \|z\|_{H^2(\Omega)}\lesssim\delta_h\|u-u_h\|_{\Omega_h},
\end{equation*}
using this result with \eqref{bound6} and \eqref{bound2} we have
\begin{equation*}
\begin{split}
|\langle& \nabla (u-u_h) \cdot n_h,z\rangle_{\Gamma_h}-\langle\nabla z\cdot n_h,T_{1,k}(u-u_h)\rangle_{\Gamma_h}|\\
&\leq\|\nabla (u-u_h) \cdot n_h\|_{\Gamma_h}\| z\|_{\Gamma_h}+\|\nabla z\cdot n_h\|_{\Gamma_h}\| T_{1,k}(u-u_h)\|_{\Gamma_h}\\
&\lesssim (h^{p}\|u\|_{H^1(\Omega)}+\vvvert u-u_h\vvvert_h)h^{-\frac12}\| z\|_{\Gamma_h}+\|z\|_{H^2(\Omega_h)}\| T_{1,k}(u-u_h)\|_{\Gamma_h}\\
&\lesssim(h^{-\frac32}\delta_h(h^{p+1}\|u\|_{H^1(\Omega)}+h\vvvert u-u_h\vvvert_h)+\| T_{1,k}(u-u_h)\|_{\Gamma_h})\|u-u_h\|_{\Omega_h}\\
&\lesssim ((h^{-\frac12}\delta_h)\vvvert u-u_h\vvvert_h+h^{p+1}\|u\|_{H^{p+1}(\Omega)})\|u-u_h\|_{\Omega_h}.
\end{split}
\end{equation*}
Also
\begin{equation*}
\begin{split}
|\langle\nabla z\cdot n_h,u-u_h\rangle_{\Gamma_h}|
&\leq \|\nabla z\cdot n_h\|_{\Gamma_h}\|u-u_h\|_{\Gamma_h}\\
&\leq \|z\|_{H^2(\Omega_h)}h^{\frac12}\vvvert u-u_h\vvvert_h\\
&\leq \|u-u_h\|_{\Omega_h}h^{\frac12}\vvvert u-u_h\vvvert_h.
\end{split}
\end{equation*}
Using $\delta_h\lesssim h^2$ we obtain the bound
\begin{equation*}
\begin{split}
&\langle \nabla (u-u_h) \cdot n_h,z\rangle_{\Gamma_h}-\langle\nabla z\cdot n_h,T_{1,k}(u-u_h)\rangle_{\Gamma_h}-2\langle\nabla z\cdot n_h,u-u_h\rangle_{\Gamma_h}\\
&\lesssim \Big(h^{p+1}\|u\|_{H^{p+1}(\Omega)}+h^\frac12\vvvert u-u_h\vvvert_h\Big)\|u-u_h\|_{\Omega_h}\\
&\lesssim\Big(h^{p+\frac12}\|u\|_{H^{p+1}(\Omega)}+\delta_h^{k+1}\underset{0\leq t\leq\delta_0}{\textup{sup}}\|D^{k+1}u\|_{\Gamma_t}\\
&\quad\quad\quad\quad\quad\quad\quad\quad\quad\quad\quad+ h^{\frac12}\delta_h^\frac12\delta_h^{l+\frac12} \underset{0\leq t\leq\delta_0}{\textup{sup}} \|D_n^l(f+\Delta u)\|_{\Gamma_t}\Big)\|u-u_h\|_{\Omega_h}.
\end{split}
\end{equation*}
The Theorem follows.
 \end{proof}

\section{Numerical Results and Discussion}
We will consider 3 examples of increasing complexity to corroborate the theoretical findings in the previous sections. 
The exact boundary of the domain $\Omega$ is described using analytical expressions of level set functions whose zero level set describes the boundary. We first consider a circular domain and a domain with convex and concave boundaries with zero dirichlet boundary conditions and then a flower shape domain with non-zero Dirichlet boundary conditions. We will demonstrate the effect of the boundary value correction terms for polynomial order 2 and 3. In all examples, we set the ghost penalty parameter to $\gamma_p=0.1$.

\subsection{Reference Solution in Circle with Zero Dirichlet Boundary Conditions}
In our first example, we consider a circular domain described by the zero level set of 
\begin{equation*}
\phi=R-1
\end{equation*}
where $R = \sqrt{x^2+y^2}$. 
We investigate the convergence of the numerical solution to the  following analytical solution 
\begin{equation*}
u(x,y)=\cos(\pi\frac{R^2}{2}),
\end{equation*}
which we prescribed using
\begin{equation*}
f(x,y)=\pi^{2} R^{2}  \cos{\left (\pi \left(\frac{R^{2}}{2} \right) \right )} + 2 \pi \sin{\left (\pi \left(\frac{R^{2}}{2} \right) \right )}.
\end{equation*}
The solution and the linear approximation of the domain are depicted in Figure~\ref{fig: ref circle sol}. Figure~\ref{fig: circleconvergence} shows the convergence rates of the numerical solution in the $H^1$ and $L^2$-norm. The order of convergence is optimal when a Taylor expansion of first order is used ($k=1$). Adding terms beyond the first order term in the Taylor expansion does not yield any improvment in the rate of convergence ($k>1$). 
\begin{figure}[h!]
\includegraphics[width=0.5\textwidth]{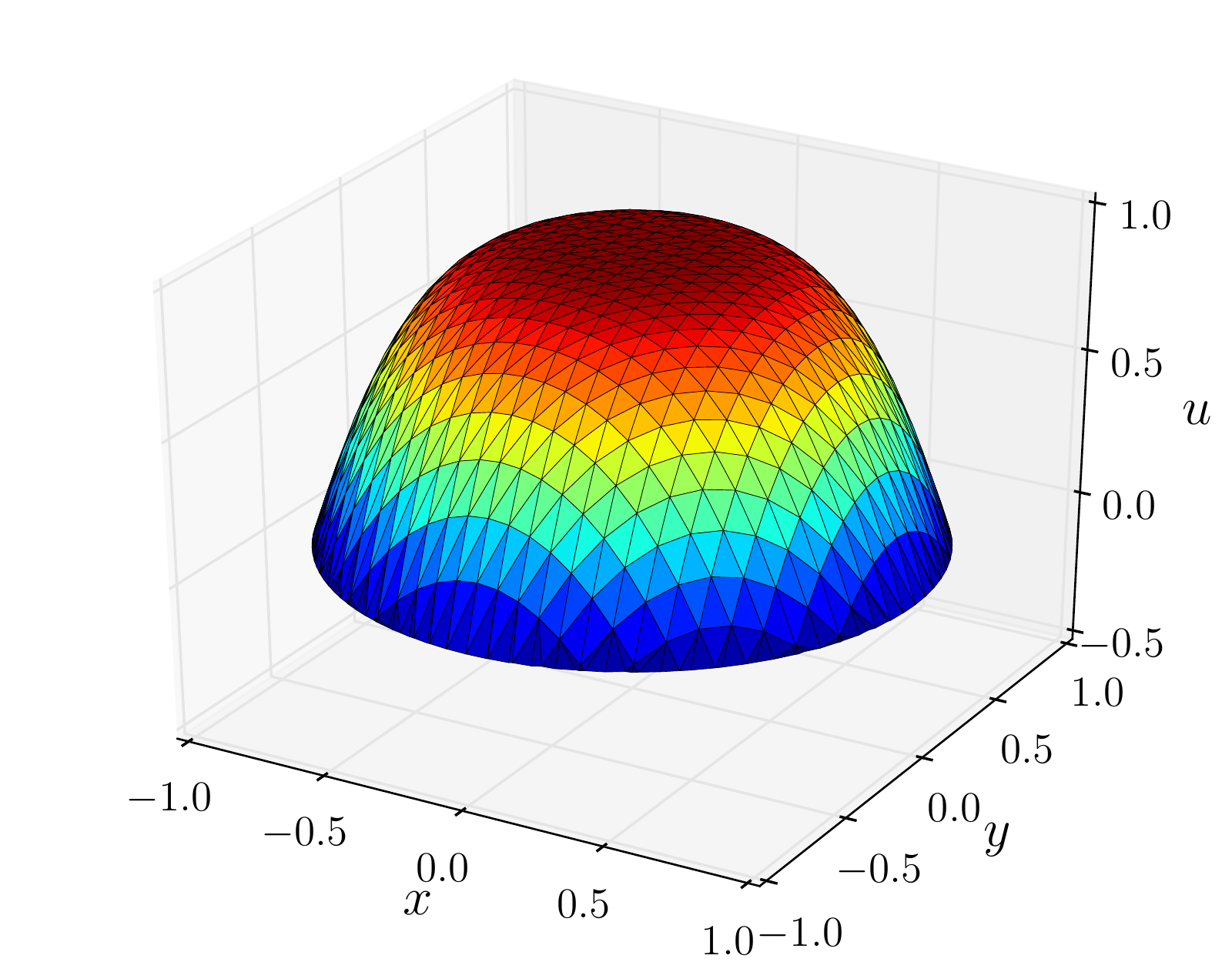}
\includegraphics[width=0.5\textwidth]{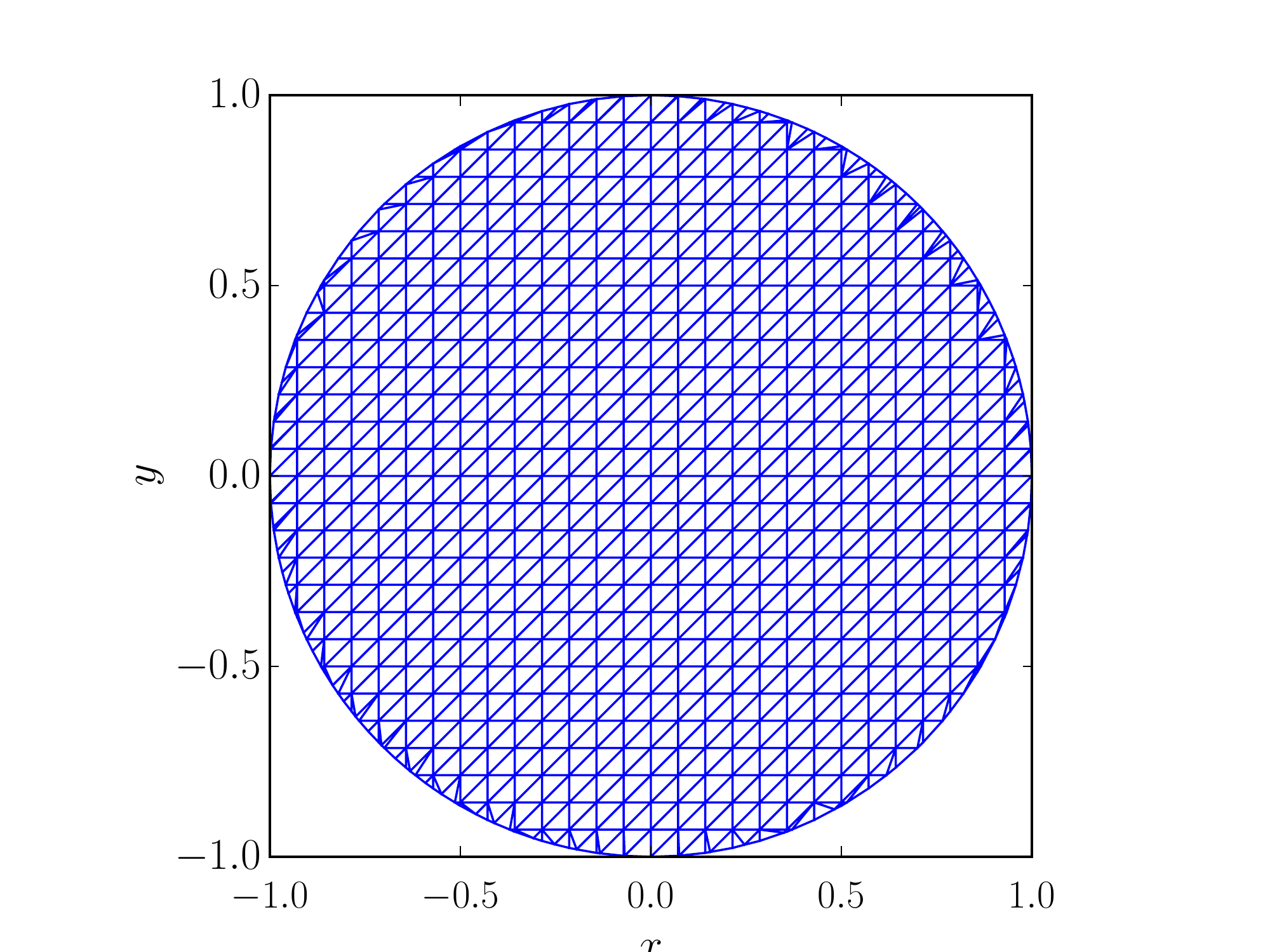}
\caption{Reference solution $u$ and the cut finite element mesh in the circle geometry.}
\label{fig: ref circle sol}
\end{figure}
\begin{figure}[h!]
\subfloat[a][$p=2$]{
\includegraphics[width=0.5\textwidth]{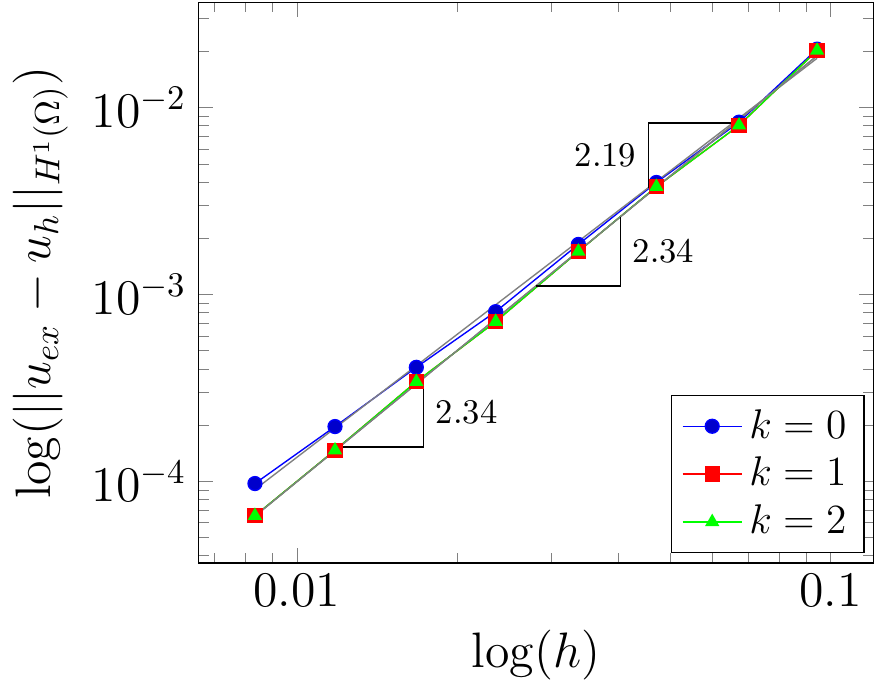}\label{subfig: p2h1circle}}
\subfloat[a][$p=3$]{\includegraphics[width=0.5\textwidth]{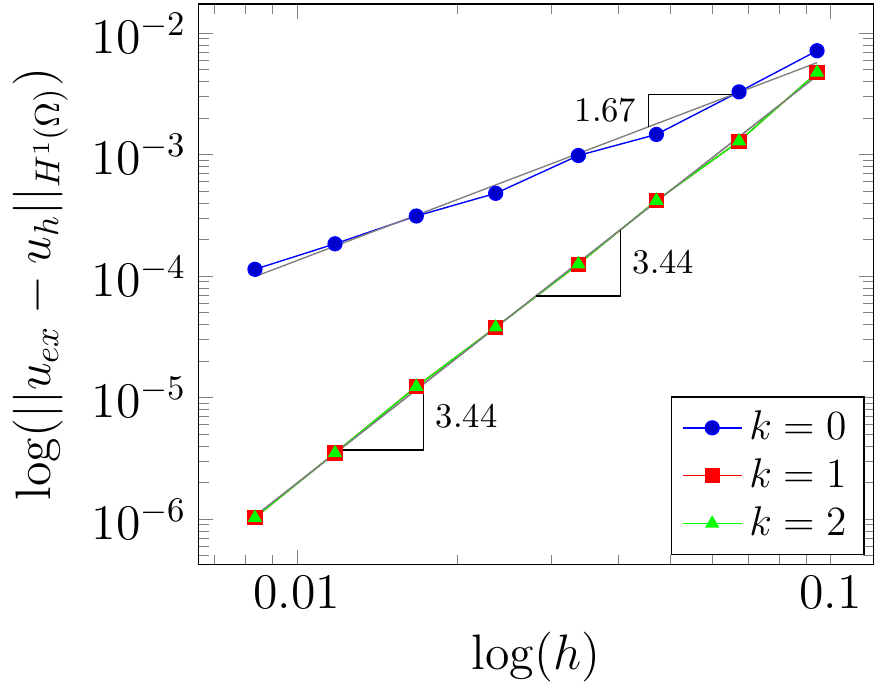}\label{subfig: p3h1circle}} \\
\subfloat[a][$p=2$]{
\includegraphics[width=0.5\textwidth]{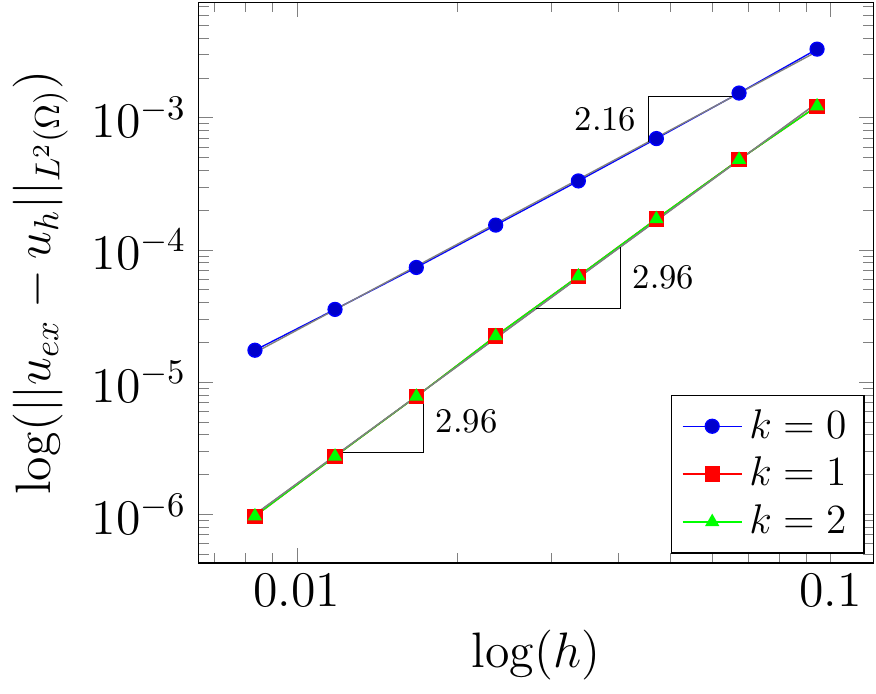}\label{subfig: p2l2circle}}
\subfloat[a][$p=3$]{\includegraphics[width=0.5\textwidth]{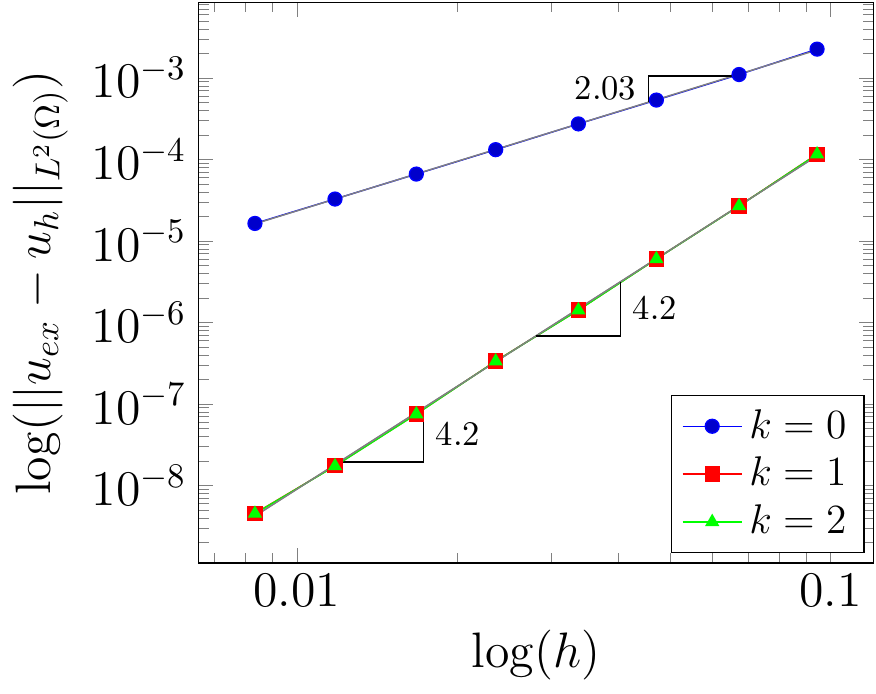}\label{subfig: p3l2circle}} \\
\caption{Convergence rates for the first reference solution in the circle in the $H_1$ norm for $p=2$, \protect\subref{subfig: p3h1circle}, $p=3$ for Taylor expansions of order $k=0,1,2$ and in the $L_2$ norm for \protect\subref{subfig: p2l2circle} $p=2$,\protect\subref{subfig: p3l2circle} $p=3$.}
\label{fig: circleconvergence}
\end{figure}
\pagebreak
\subsection{Reference Solution in Torus with Zero Dirichlet Boundary Conditions}
Next, we consider a domain with convex and concave boundaries given by the zero level set of the function
\begin{equation}
\phi=\left(R - 0.75\right) \left(R - 0.25\right).
\end{equation}
We set 
\begin{align}
f =& 20 \left(4 - \frac{1}{R} \right)
\end{align}
and obtain the analytical solution 
\begin{equation}
u=20\left(0.75 - R \right) \left(R - 0.25\right). 
\end{equation}
as shown in Figure~\ref{fig: ref torus sol}. The convergence rates shown in Figure~\ref{fig: torusconvergence} are optimal for $p=2$, $p=3$ when a first order Taylor expansion is used in the boundary value correction terms. 
\begin{figure}[h!]
\subfloat[]{
\includegraphics[width=0.5\textwidth]{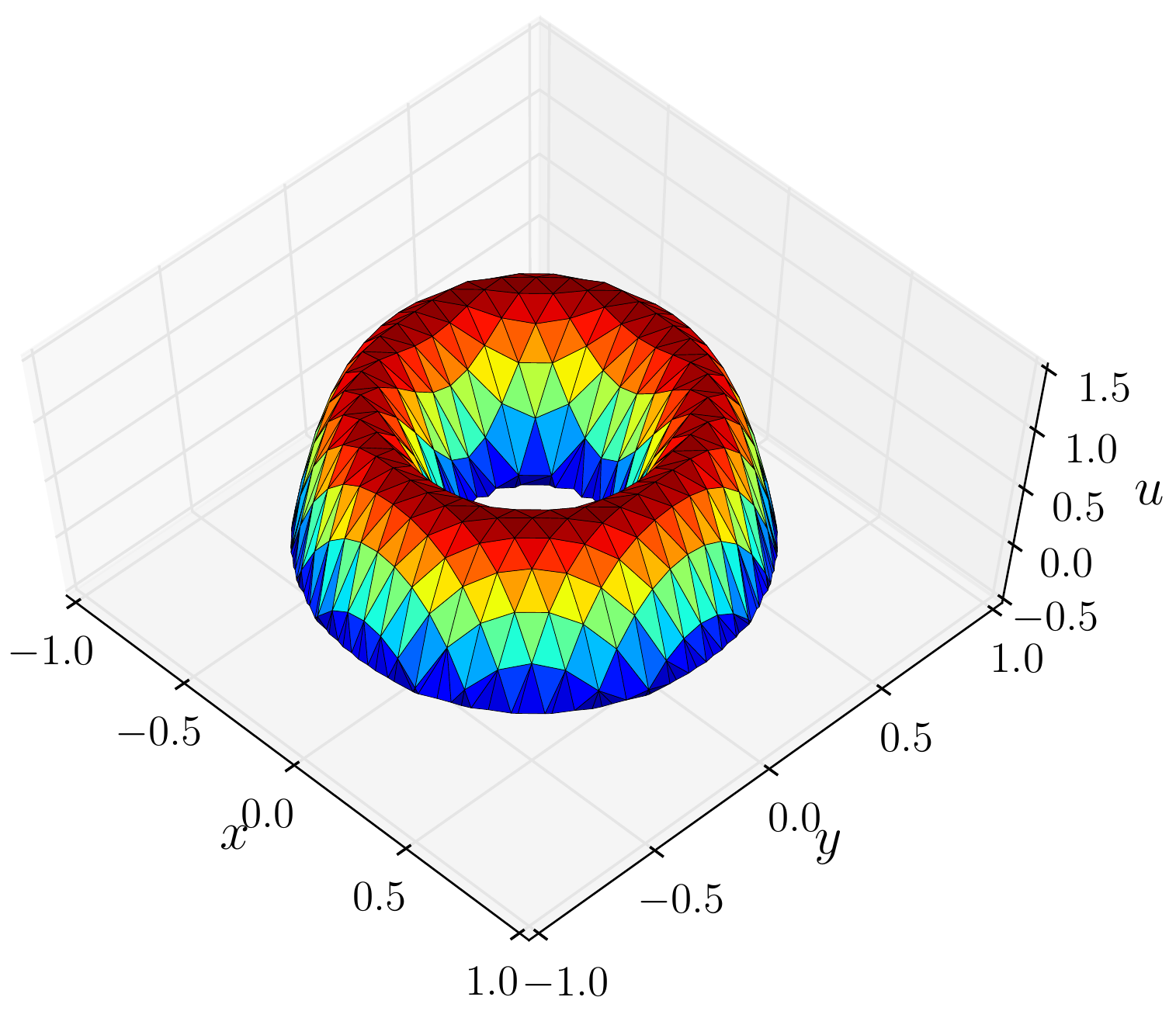}\label{subfig: torus sol}}
\subfloat[]{\includegraphics[width=0.5\textwidth]{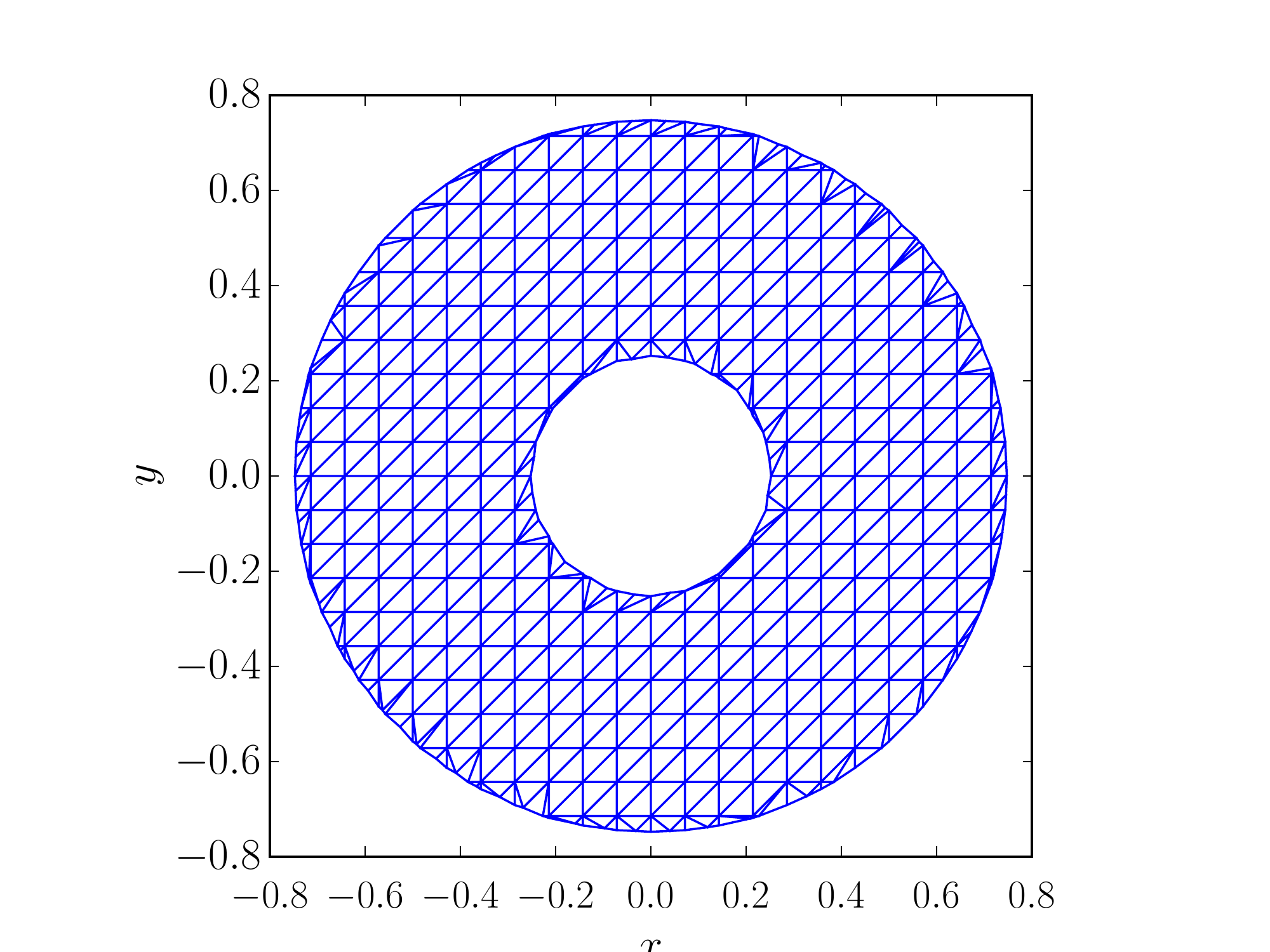}\label{subfig: torus mesh}}
\caption{Reference solution $u$ in a torus-like shaped geometry and the cut finite element mesh of the domain.}
\label{fig: ref torus sol}
\end{figure}
\begin{figure}[h!]
\subfloat[a][$p=2$]{
\includegraphics[width=0.5\textwidth]{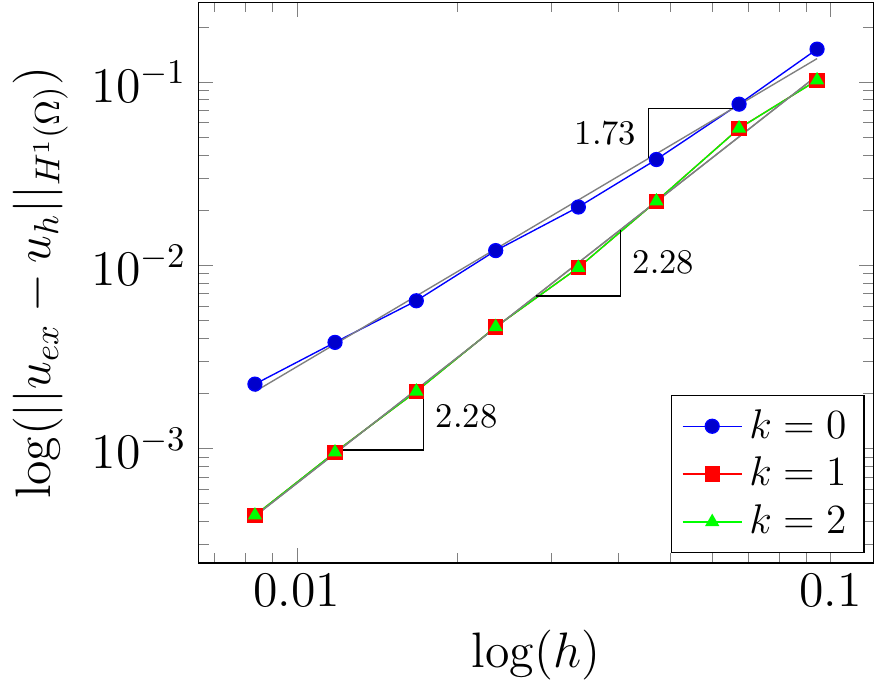}\label{subfig: p2h1torus}}
\subfloat[a][$p=3$]{\includegraphics[width=0.5\textwidth]{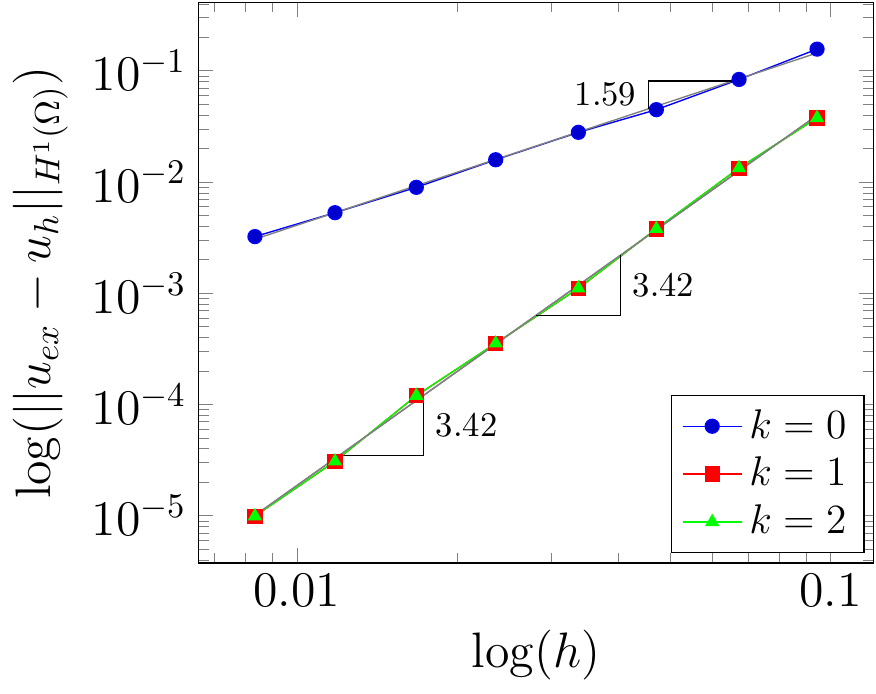}\label{subfig: p3h1torus}} \\
\subfloat[a][$p=2$]{
\includegraphics[width=0.5\textwidth]{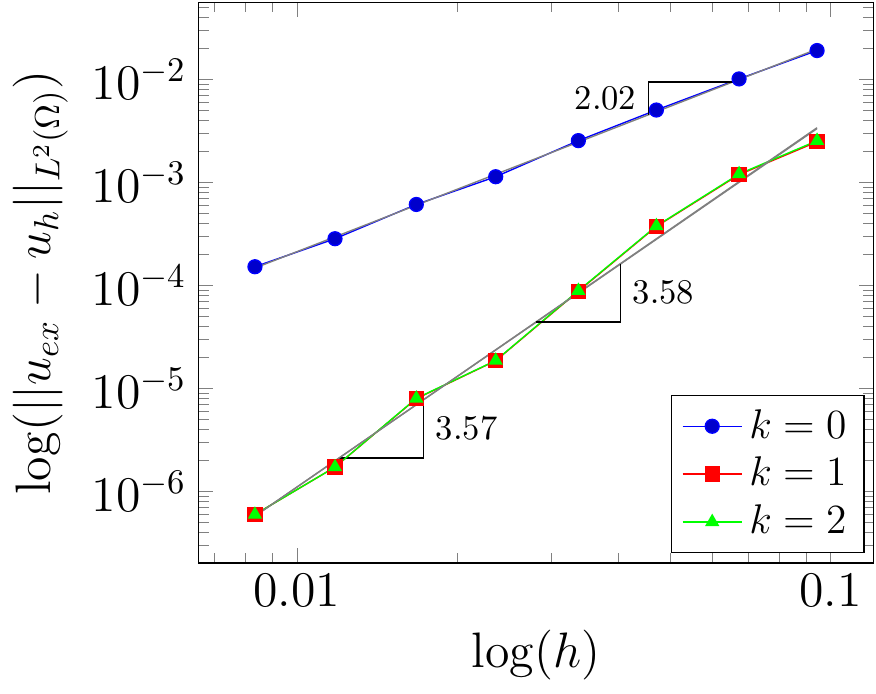}\label{subfig: p2l2torus}}
\subfloat[a][$p=3$]{\includegraphics[width=0.5\textwidth]{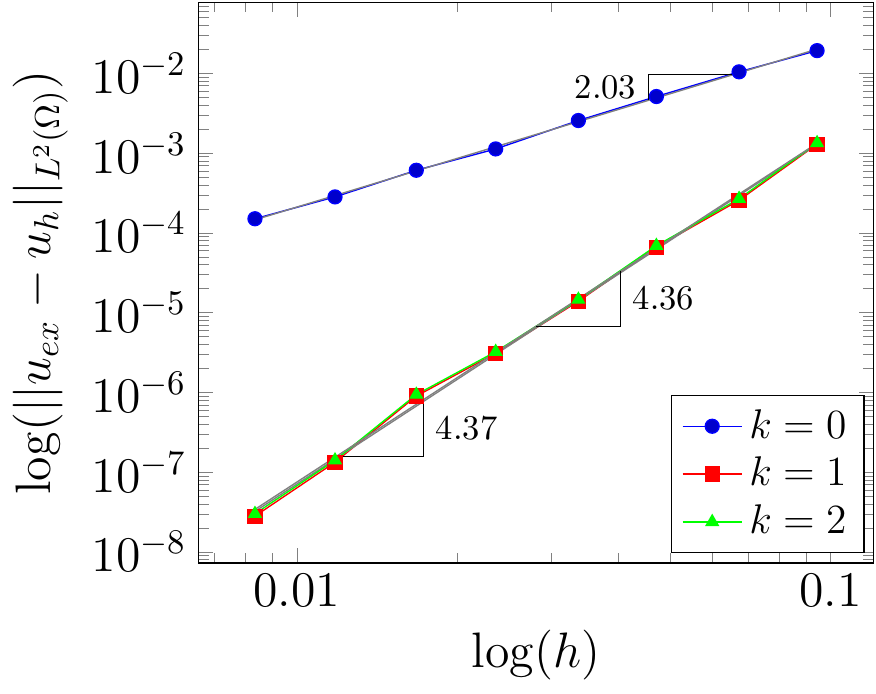}\label{subfig: p3l2torus}} \\
\caption{Convergence rates for the reference solution in the domain including convex and concave boundaries in the $H_1$ norm for \protect\subref{subfig: p2h1torus} $p=2$, \protect\subref{subfig: p3h1torus}, $p=3$ for Taylor expansions of order $k=0,1,2$ and in the $L_2$ norm for \protect\subref{subfig: p2l2torus} $p=2$,\protect\subref{subfig: p3l2torus} $p=3$.}
\label{fig: torusconvergence}
\end{figure}

%

\subsection{Reference Solution in Flower Shape with Non-Zero Dirichlet Boundary Conditions}
In our final example, we consider a flower like shaped domain ~\cite{lehrenfeld2016high} defined by
\begin{equation}
\phi = (R^2-r_{\theta})(R^2-(1.0/6.0)^2)
\end{equation}
with $r_{\theta} = r_0 + 0.1 \sin(\omega \theta)$, 
$r_0 = 1/2$, $\omega = 8$ and $\theta=\arctan(x/y)$. We investigate the convergence rates of our numerical solution 
with respect to 
\begin{equation}
u=\cos{\left ( \pi \frac{x}{2} \right )} \cos{\left ( \pi \frac{y}{2} \right )}
\end{equation}
\begin{equation}
f=\frac{\pi^{2}}{2} \cos{\left ( \pi \frac{x}{2} \right )} \cos{\left ( \pi \frac{y}{2} \right )}
\end{equation}
The reference solution and the cut mesh are shown in Figure~\ref{fig: ref flower sol}. Figure~\ref{fig: flowerconvergence} shows the convergence rate for $p=2$ and $p=3$. 
\begin{figure}[h!]
\includegraphics[width=0.5\textwidth]{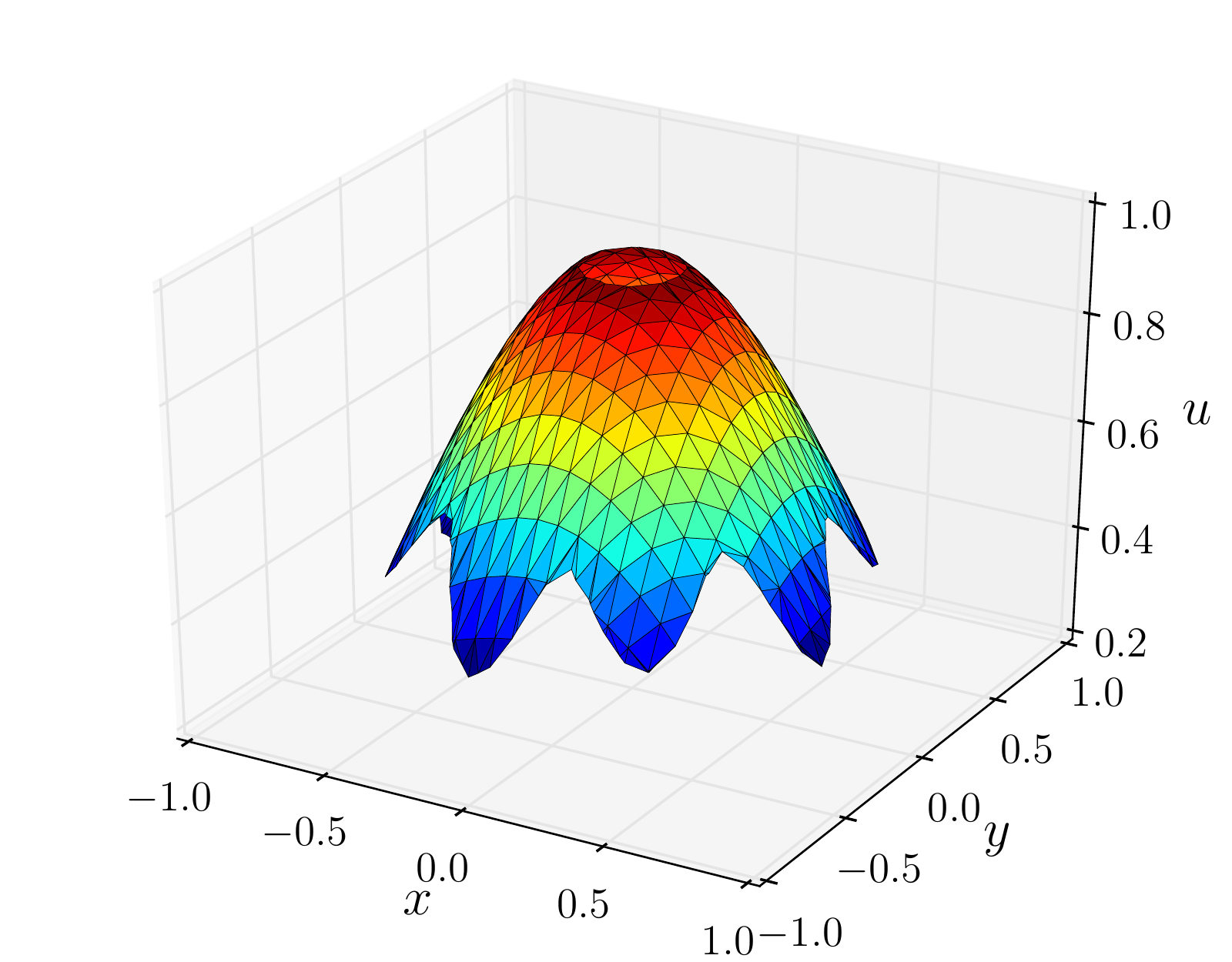}
\includegraphics[width=0.5\textwidth]{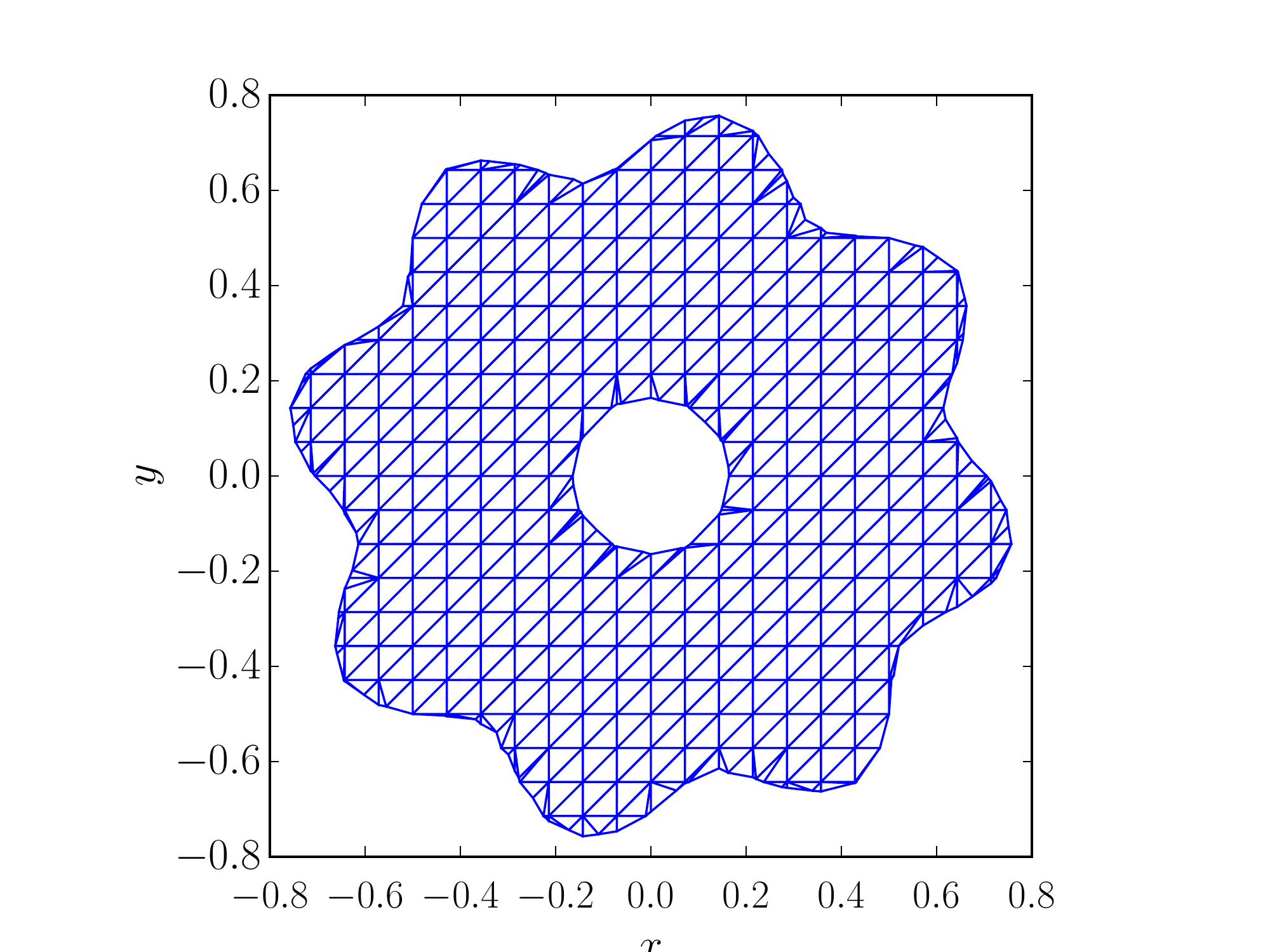}
\caption{Reference solution $u$ in flower geometry and the cut finite element mesh of the domain.}
\label{fig: ref flower sol}
\end{figure}
\begin{figure}[h!]
\subfloat[a][$p=2$]{
\includegraphics[width=0.5\textwidth]{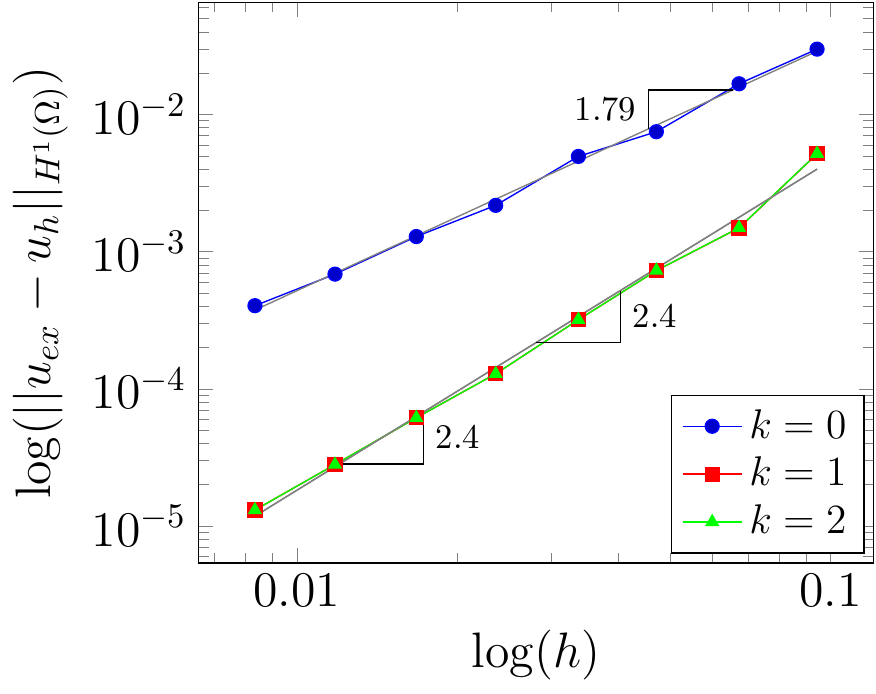}\label{subfig: p2h1flower}}
\subfloat[a][$p=3$]{\includegraphics[width=0.5\textwidth]{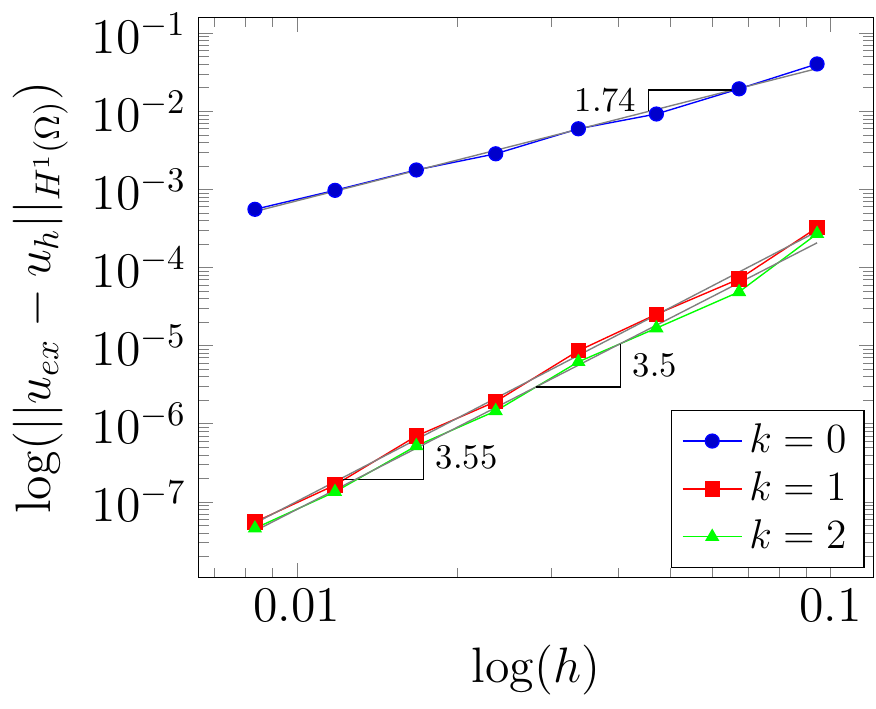}\label{subfig: p3h1flower}} \\
\subfloat[a][$p=2$]{
\includegraphics[width=0.5\textwidth]{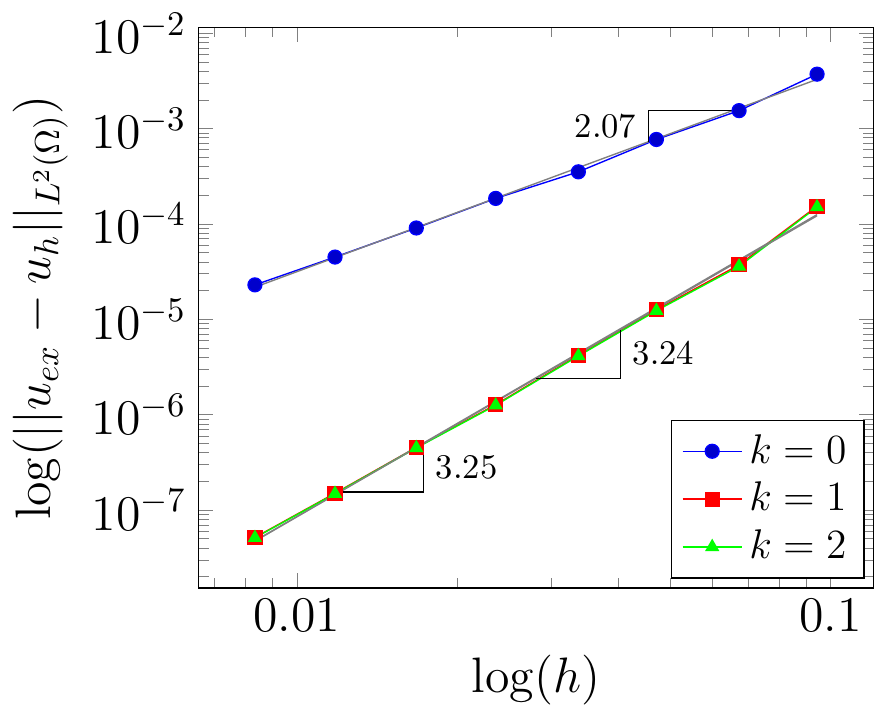}\label{subfig: p2l2flower}}
\subfloat[a][$p=3$]{\includegraphics[width=0.5\textwidth]{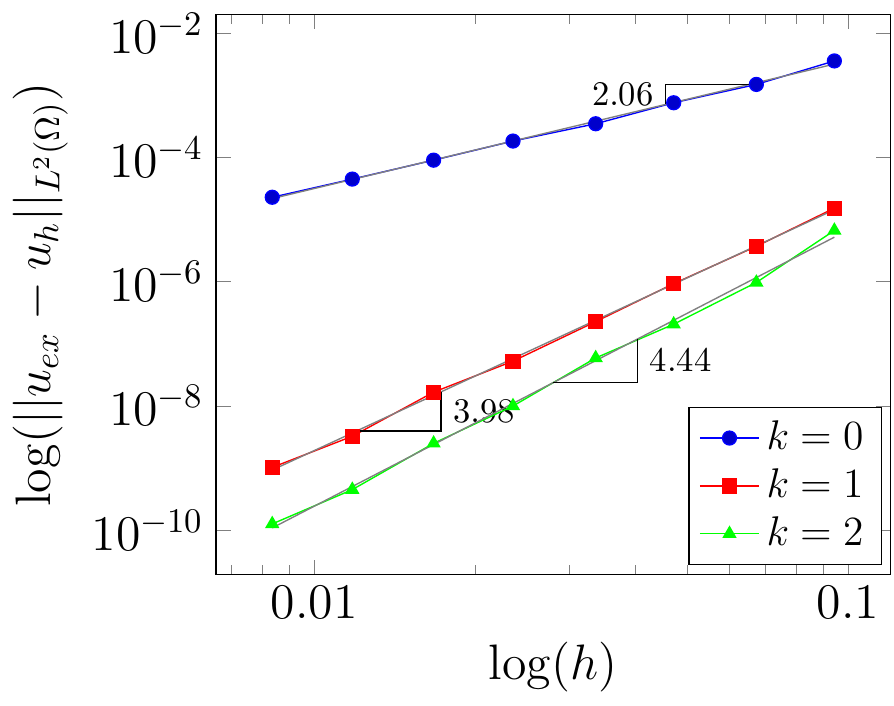}\label{subfig: p3l2flower}} \\
\caption{Convergence rates for the reference solution in the flower shaped domain in the $H_1$ norm for \protect\subref{subfig: p2h1flower} $p=2$, \protect\subref{subfig: p3h1flower}, $p=3$ for Taylor expansions of order $k=0,1,2$ and in the $L_2$ norm for \protect\subref{subfig: p2l2flower} $p=2$,\protect\subref{subfig: p3l2flower} $p=3$.}
\label{fig: flowerconvergence}
\end{figure}

\subsection{3D Solution in an Ellipsoid}
\begin{figure}[th]
\includegraphics[width=0.6\textwidth]{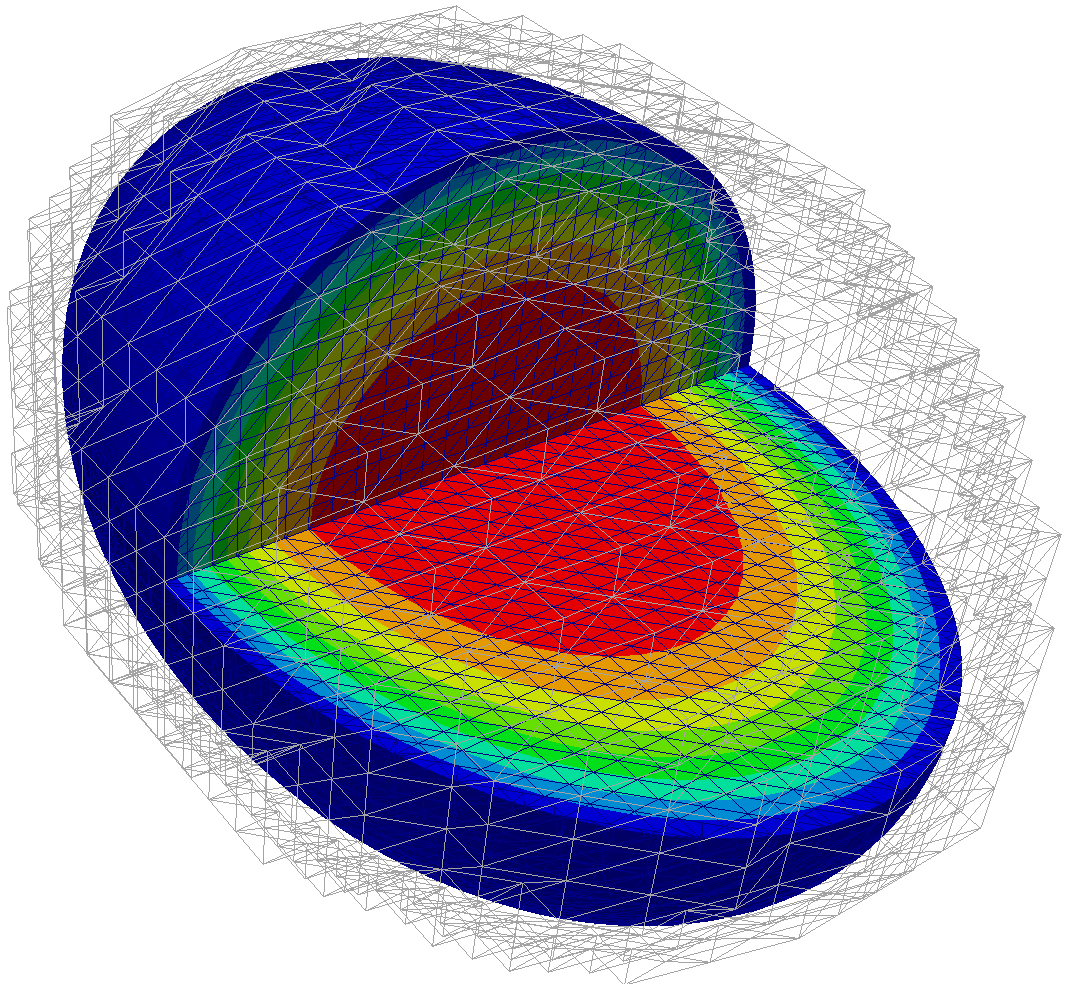}
\includegraphics[width=0.2\textwidth]{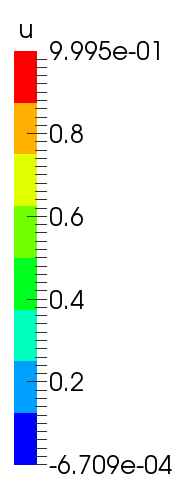}
\caption{Reference solution $u$ in a 3D ellipsoid.}
\label{fig: ref ellipsoid sol}
\end{figure}
We compute the reference solution 
\begin{equation}
u = \cos(\pi \frac{\hat{r}^2}{2})
\end{equation}
with 
\begin{equation}
\hat{r} = \sqrt{\frac{x^2}{(3.0/4.0)^2}+\frac{y^2}{(1.0/2.0)^2}+\frac{z^2}{(1.0/2.0)^2}}
\end{equation}
in a 3D ellipsoid given by the function 
\begin{equation}
 \phi = \hat{r}-1.
\end{equation}   
Figure~\ref{fig: ref ellipsoid sol} shows the solution in the ellipsoid and Figure~\ref{fig: ellipsoidconvergence} shows the convergence for the solution for $p=2$ and $k=1$ demonstrating the optimal convergence rate of the numerical solution as predicted by the estimates of the previous section. 
\begin{figure}[h!]
\centering
{
\includegraphics[width=0.5\textwidth]{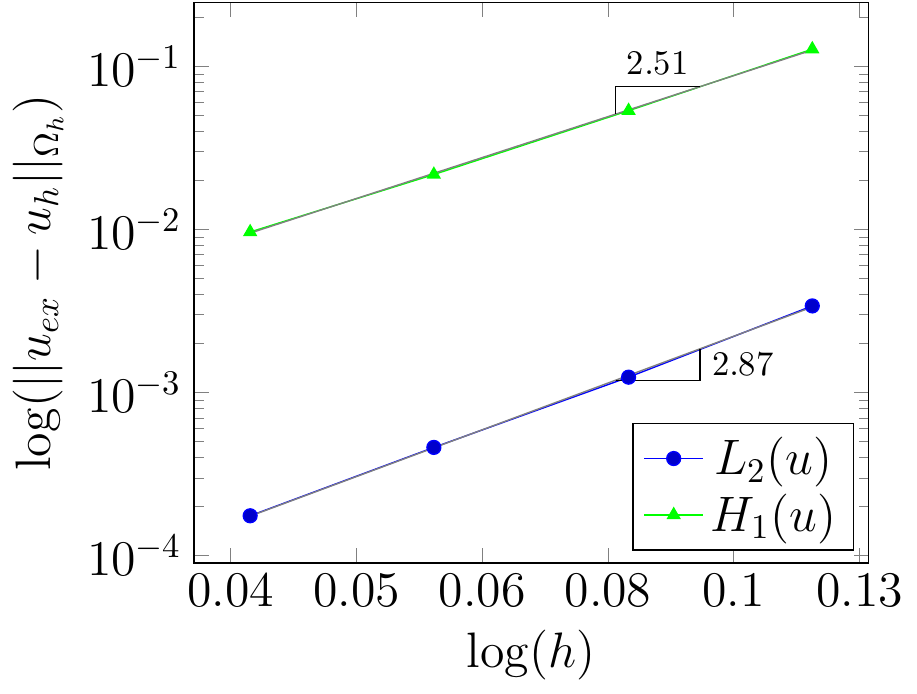}\label{subfig: p2h1ellipsoid3D}}
\caption{Error for $p=2$ in ellipsoid for $k=1$.}
\label{fig: ellipsoidconvergence}
\end{figure} 

%
%

	\bibliography{Bibliography}
\end{document}

%% file: mesh_domain.tex
\begin{tikzpicture}[scale=0.7]
\draw [fill=white,white] (-0.4,-0.4) rectangle (0.4,0.4); 

\draw  [fill=black, fill opacity=0.2] (0,0) circle (3.1cm);
       
\draw [black] (0,0)node{\LARGE{$\Omega$}};
	\begin{pgfonlayer}{nodelayer}
		\node [style=none] (0) at (-4, 4) {};
		\node [style=none] (1) at (4, 4) {};
		\node [style=none] (2) at (4, -4) {};
		\node [style=none] (3) at (-4, -4) {};
		\node [style=none] (4) at (-3.1, 0) {};
		\node [style=none] (5) at (0, 3.1) {};
		\node [style=none] (6) at (3.1, 0) {};
		\node [style=none] (7) at (0, -3.1) {};
		\node [style=none] (8) at (-3, 4) {};
		\node [style=none] (9) at (-2, 4) {};
		\node [style=none] (10) at (-1, 4) {};
		\node [style=none] (11) at (0, 4) {};
		\node [style=none] (12) at (1, 4) {};
		\node [style=none] (13) at (2, 4) {};
		\node [style=none] (14) at (3, 4) {};
		\node [style=none] (15) at (4, 3) {};
		\node [style=none] (16) at (4, 2) {};
		\node [style=none] (17) at (4, 1) {};
		\node [style=none] (18) at (4, 0) {};
		\node [style=none] (19) at (4, -1) {};
		\node [style=none] (20) at (4, -2) {};
		\node [style=none] (21) at (4, -3) {};
		\node [style=none] (22) at (3, -4) {};
		\node [style=none] (23) at (2, -4) {};
		\node [style=none] (24) at (1, -4) {};
		\node [style=none] (25) at (0, -4) {};
		\node [style=none] (26) at (-1, -4) {};
		\node [style=none] (27) at (-2, -4) {};
		\node [style=none] (28) at (-3, -4) {};
		\node [style=none] (29) at (-4, -3) {};
		\node [style=none] (30) at (-4, -2) {};
		\node [style=none] (31) at (-4, -1) {};
		\node [style=none] (32) at (-4, 0) {};
		\node [style=none] (33) at (-4, 1) {};
		\node [style=none] (34) at (-4, 2) {};
		\node [style=none] (35) at (-4, 3) {};
	\end{pgfonlayer}
	\begin{pgfonlayer}{edgelayer}
		\draw (0.center) to (3.center);
		\draw (3.center) to (2.center);
		\draw (2.center) to (1.center);
		\draw (0.center) to (1.center);
		\draw [style=circle, bend right=45, black, very thick] (4.center) to (7.center);
		\draw [style=circle, bend right=45, black, very thick] (5.center) to (4.center);
		\draw [style=circle, bend left=45, black, very thick] (5.center) to (6.center);
		\draw [style=circle, bend left=45, black, very thick] (6.center) to (7.center);
		\draw [style=circle] (35.center) to (8.center);
		\draw [style=circle] (34.center) to (9.center);
		\draw [style=circle] (33.center) to (10.center);
		\draw [style=circle] (32.center) to (11.center);
		\draw [style=circle] (31.center) to (12.center);
		\draw [style=circle] (30.center) to (13.center);
		\draw [style=circle] (29.center) to (14.center);
		\draw [style=circle] (1.center) to (3.center);
		\draw [style=circle] (28.center) to (15.center);
		\draw [style=circle] (16.center) to (27.center);
		\draw [style=circle] (26.center) to (17.center);
		\draw [style=circle] (18.center) to (25.center);
		\draw [style=circle] (24.center) to (19.center);
		\draw [style=circle] (20.center) to (23.center);
		\draw [style=circle] (22.center) to (21.center);
		\draw [style=circle] (14.center) to (22.center);
		\draw [style=circle] (13.center) to (23.center);
		\draw [style=circle] (24.center) to (12.center);
		\draw [style=circle] (11.center) to (25.center);
		\draw [style=circle] (26.center) to (10.center);
		\draw [style=circle] (28.center) to (8.center);
		\draw [style=circle] (9.center) to (27.center);
		\draw [style=circle] (35.center) to (15.center);
		\draw [style=circle] (16.center) to (34.center);
		\draw [style=circle] (33.center) to (17.center);
		\draw [style=circle] (32.center) to (18.center);
		\draw [style=circle] (31.center) to (19.center);
		\draw [style=circle] (30.center) to (20.center);
		\draw [style=circle] (29.center) to (21.center);
	\end{pgfonlayer}
\end{tikzpicture}

%% file: patch_structured.tex
\begin{tikzpicture}[scale=0.8]

%
%
%
%
%
%

\draw [black] (-0.7,-1.3) node{$K_1$};
\draw [black] (-1.5,-0.5) node{$K_2$};
\draw [black] (-0.8,0.7) node{$K_3$};
\draw [black] (0.6,1.6) node{$K_4$};
\draw [black] (1.3,0.6) node{$K_5$};
\draw [black] (2.6,1.3) node{$K_6$};
\draw [black] (3.4,2.7) node{$K_7$};

	\begin{pgfonlayer}{nodelayer}
		\node [style=nfilled] (0) at (-2, -2) {};
		\node [style=filled] (1) at (-2, 0) {};
		\node [style=filled] (2) at (0, 2) {};
		\node [style=filled] (3) at (2, 2) {};
		\node [style=filled] (4) at (2, 0) {};
		\node [style=filled] (5) at (0, 0) {};
		\node [style=nfilled] (6) at (0, -2) {};
		\node [style=nfilled] (7) at (4, 4) {};
		\node [style=nfilled] (8) at (4, 2) {};
		\node [style=nfilled] (9) at (4, 0) {};
		\node [style=none] (10) at (-1.85, -2) {};
		\node [style=none] (11) at (4, 2.3) {};
		\node [style=nfilled] (12) at (2, -2) {};
		\node [style=none] (13) at (-1.80, -1.8) {};
		\node [style=none] (14) at (-0.62, 0) {};
		\node [style=none] (15) at (0, 0.62) {};
		\node [style=none] (16) at (1.8, 1.8) {};
		\node [style=none] (17) at (2, 1.88) {};
		\node [style=none] (18) at (2.45, 2) {};
		\node [style=none] (19) at (4, 2.28) {};
	\end{pgfonlayer}
	\begin{pgfonlayer}{edgelayer}
		\draw [style=none] (0) to (5);
		\draw [style=none] (1) to (2);
		\draw [style=none] (5) to (3);
		\draw [style=none] (6) to (4);
		\draw [style=none] (4) to (8);
		\draw [style=none] (3) to (7);
		\draw [style=none] (2) to (3);
		\draw [style=none] (1) to (0);
		\draw [style=none] (0) to (6);
		\draw [style=none] (5) to (6);
		\draw [style=none] (1) to (5);
		\draw [style=none] (2) to (5);
		\draw [style=none] (5) to (4);
		\draw [style=none] (3) to (4);
		\draw [style=none] (3) to (8);
		\draw [style=none] (7) to (8);
		\draw [style=none] (8) to (9);
		\draw [style=none] (4) to (9);
		\draw [style=none, black, dashed, thick]  (10) to (13);
		\draw [style=none, black, dashed, thick]  (13) to (14);
		\draw [style=none, black, dashed, thick]  (14) to (15);
		\draw [style=none, black, dashed, thick]  (15) to (16);
		\draw [style=none, black, dashed, thick]  (16) to (17);
		\draw [style=none, black, dashed, thick]  (17) to (18);
		\draw [style=none, black, dashed, thick]  (18) to (19);
		\draw [style=none] (6) to (12);
		\draw [style=none] (4) to (12);
		\draw [style=none] (12) to (9);
	\end{pgfonlayer}
\end{tikzpicture}